\documentclass[11pt, a4paper]{article}

\usepackage[english]{babel}
\usepackage[a4paper]{geometry}
\usepackage{amsmath,amssymb}
\usepackage{amsfonts}
\usepackage{amsthm}
\usepackage{graphicx}
\usepackage[utf8]{inputenc}
\usepackage{mathtools}
\usepackage{paralist}
\usepackage{booktabs}

\numberwithin{equation}{section}

\theoremstyle{plain}
\newtheorem{theorem}{Theorem}[section]

\newtheorem{lemma}[theorem]{Lemma}
\newtheorem{corollary}[theorem]{Corollary}

\theoremstyle{definition}

\newtheorem{remark}[theorem]{Remark}

\newcommand{\K}{\mathbb{K}}
\newcommand{\C}{\mathbb{C}}
\newcommand{\R}{\mathbb{R}}

\newcommand{\N}{\mathbb{N}}

\newcommand{\cC}{\mathcal{C}}

\newcommand{\cL}{\mathcal{L}}

\newcommand{\coloneq}{\mathrel{\mathop:}=}
\newcommand{\eqcolon}{=\mathrel{\mathop:}}

\newcommand{\rmd}{\textrm{d}}

\newcommand{\ii}{{\operatorname{i}}}

\newcommand{\conj}[1]{\overline{#1}}

\DeclarePairedDelimiter{\abs}{\lvert}{\rvert}
\DeclarePairedDelimiter{\norm}{\lVert}{\rVert}
\DeclarePairedDelimiter{\sprod}{\langle}{\rangle}

\DeclarePairedDelimiter{\cc}{[}{]}

\DeclarePairedDelimiter{\oo}{]}{[}

\DeclareMathOperator{\re}{Re}
\DeclareMathOperator{\im}{Im}

\DeclareMathOperator{\diag}{diag}
\DeclareMathOperator{\spann}{span}

\newcommand{\ltwospace}{\mathcal{L}^2(\Xi,\rho)}

\author{Roland Pulch\footnotemark[1] \and Olivier S\`{e}te\footnotemark[2]}
\title{The Helmholtz equation with uncertainties in the wavenumber}
\date{September 29, 2022}

\begin{document}
\maketitle

\renewcommand{\thefootnote}{\fnsymbol{footnote}}

\footnotetext[1]{Institute of Mathematics and Computer Science, Universit\"at 
Greifswald, Walther-Rathenau-Stra\ss{}e~47, 17489 Greifswald, Germany.
\texttt{roland.pulch@uni-greifswald.de}}

\footnotetext[2]{Institute of Mathematics and Computer Science, Universit\"at 
Greifswald, Walther-Rathenau-Stra\ss{}e~47, 17489 Greifswald, Germany.
\texttt{olivier.sete@uni-greifswald.de}.
ORCID: 0000-0003-3107-3053}

\renewcommand{\thefootnote}{\arabic{footnote}}


\begin{abstract}
We investigate the Helmholtz equation with suitable boundary conditions and 
uncertainties in the wavenumber. 
Thus the wavenumber is modeled as a random variable or a random field. 
We discretize the Helmholtz equation using finite differences in space, which 
leads to a linear system of algebraic equations including random variables. 
A stochastic Galerkin method yields a deterministic linear system of algebraic 
equations. 
This linear system is high-dimensional, sparse and complex symmetric but, in 
general, not hermitian.
We therefore solve this system iteratively with GMRES and propose two 
preconditioners: a complex shifted Laplace preconditioner and a mean value 
preconditioner.
Both preconditioners reduce the number of iteration steps as well as the 
computation time in our numerical experiments.
\end{abstract}




\section{Introduction}

The Helmholtz equation is a linear partial differential equation (PDE), whose 
solutions are time-harmonic states of the wave equation, 
see~\cite{GriffithsDoldSilvester2015,Polyanin2002}. 
Important applications of this model are given in acoustics and 
electromagnetics~\cite{ColtonKress}. 
The Helmholtz equation includes a wavenumber, which is either a constant 
parameter or a space-dependent function. 
Furthermore, boundary conditions are imposed on the spatial domain. 

We consider uncertainties in the wavenumber. 
Thus the wavenumber is replaced by a random variable or a spatial random field 
to quantify the uncertainties. 
The solution of the Helmholtz equation changes into a random field,  which can 
be expanded into the (generalized) polynomial chaos, see~\cite{Xiu2010}. 
We employ the stochastic Galerkin method to compute approximations of the 
unknown coefficient functions. 
Stochastic Galerkin methods were used for linear PDEs of different types 
including random variables, for example, 
see~\cite{Gittelson2013,YoussefPulch2021} on elliptic type, 
\cite{GottliebXiu,PulchXiu2012} on hyperbolic type, 
and~\cite{PulchvanEmmerich2009,XiuShen} on parabolic type.
Wang et al.~\cite{Wang-etal} applied a multi-element stochastic 
Galerkin method to solve the Helmholtz equation including random variables. 
We investigate the ordinary stochastic Galerkin method, 
which is efficient if the wavenumbers are not close to resonance. 

The stochastic Galerkin method transforms the random-dependent 
Helmholtz equation into a deterministic system of linear PDEs. 
Likewise, the original boundary conditions yield boundary conditions 
for this system. 
We examine the system of PDEs in one and two space dimensions.  
A finite difference method, see~\cite{GrossmannRoos}, 
produces a high-di\-men\-sional linear system of algebraic equations. 
When considering absorbing boundary conditions, 
the coefficient matrices are complex-valued and 
non-hermitian. 

We focus on the numerical solution of the linear systems of algebraic 
equations.  The dimension of these linear systems rapidly grows for increasing 
numbers of random variables.  Hence we use iterative methods like 
GMRES~\cite{SaadSchultz1986} in the numerical solution.
The efficiency of an iterative method strongly depends on an 
appropriate preconditioning of the linear systems.
We propose two preconditioners in the general case where the wavenumber can 
depend on space and on multiple random variables:
a complex shifted Laplace preconditioner, see~\cite{ErlanggaVuikOosterlee2004, 
GarciaNabben2018}, and a mean value preconditioner, 
see~\cite{GhanemKruger,Wang-etal}.
Statements on the location of spectra and estimates of matrix norms 
are shown.  
Furthermore, results of numerical computations are presented for both settings.

The article is organized as follows.  The stochastic Helmholtz equation is 
introduced in Section~\ref{sect:setting} and discretized in 
Section~\ref{sect:stochastic_Helmholtz}.  We discuss the complex shifted 
Laplace preconditioner in Section~\ref{sect:csl} and the mean value 
preconditioner in Section~\ref{sect:mean-value}.
Sections~\ref{sect:experiments1d} and~\ref{sect:experiments2d} contain 
numerical experiments in one and two spatial dimensions, respectively, 
which show the effectiveness of the preconditioners.
An appendix includes the detailed formulas of the discretizations in space.

\section{Problem Definition} \label{sect:setting}

We illustrate the stochastic problem associated to the Helmholtz equation. 

\subsection{Helmholtz equation}

The Helmholtz equation is a PDE of the form
\begin{equation} \label{eqn:de_orig}
- \Delta u - k^2 u = f \quad \text{in } Q  
\end{equation}
with an (open) spatial domain $Q \subseteq \R^d$ and given source term $f : 
Q \to \R$. 
The wavenumber $k$ is either a positive constant or 
a function $k : \overline{Q} \rightarrow \R_+$.
The unknown solution is $u : \overline{Q} \rightarrow \mathbb{K}$
with either $\mathbb{K} = \R$ or $\mathbb{K} = \C$.
Here $\Delta = \sum_{j=1}^d \frac{\partial^2}{\partial x_j^2}$ denotes 
the Laplace operator with respect to $x = [x_1, \ldots, x_d]^\top \in \R^d$.

Often homogeneous Dirichlet boundary conditions, i.e., 
\begin{equation} \label{eqn:bc_orig_dirichlet}
u = 0 \quad \text{on } \partial Q ,
\end{equation} 
are applied for simplicity.
Alternatively, absorbing boundary conditions read as
\begin{equation}
\partial_n u - \ii k u = 0 \quad \text{on } \partial Q ,
\label{eqn:bc_orig}
\end{equation}
where $\partial_n$ denotes the derivative with respect to the outward 
normal of $Q$ and $\ii = \sqrt{-1}$ is the imaginary unit.

\subsection{Stochastic modeling}
We consider uncertainties in the wavenumber. 
A simple model to include a variation of the wavenumber is to 
replace the constant~$k$ by a random variable on a probability 
space $(\Omega,\mathcal{A},P)$. 
We write $k=k(\xi)$, where $\xi : \Omega \rightarrow \R$ is 
some random variable with a traditional probability distribution. 
More generally, the wavenumber can be a space-dependent function 
on~$\overline{Q}$ including a multidimensional random variable 
$\xi : \Omega \rightarrow \Xi$ with $\Xi \subseteq \R^s$. 
We assume $\xi = (\xi_1,\ldots,\xi_s)^\top$ with independent 
random variables $\xi_{\ell}$ for $\ell=1,\ldots,s$.
Now the wavenumber becomes a random field
\begin{equation} \label{eqn:random_wavenumber} 
k(x,\xi) = k_0(x) + \sum_{\ell=1}^s \xi_{\ell} k_{\ell}(x) 
\end{equation}
with given functions 
$k_{\ell} : \overline{Q} \rightarrow \R$ for $\ell=0,1,\ldots,s$, 
as in~\cite{Wang-etal}. 
A truncation of a Karhunen-Lo{\`e}ve expansion, 
see~\cite[p.~17]{GhanemSpanos}, also yields a random input of 
the form~\eqref{eqn:random_wavenumber}.
Consequently, the solution of the deterministic  
Helmholtz equation~\eqref{eqn:de_orig} changes into a random field 
$u : \overline{Q} \times \Xi \rightarrow \mathbb{K}$.
We write $u(x,\xi)$ to indicate the dependence of the solution 
on space as well as the random variables. 

We assume that each random variable~$\xi_{\ell}$ has a 
probability density function~$\rho_{\ell}$. 
Since the random variables are independent, 
the product $\rho = \rho_1 \cdots \rho_s$ is the joint 
probability density function. 
Without loss of generality, let $\rho(\xi) > 0$ for almost 
all $\xi \in \Xi$.
The expected value of a measurable function $f : \Xi \rightarrow \mathbb{K}$ 
depending on the random variables is
\begin{equation*}
\mathbb{E}(f) = \int_{\Omega} f(\xi(\omega)) \; \rmd P(\omega) =
\int_{\Xi} f(\xi) \, \rho(\xi) \, \rmd \xi,
\end{equation*}
if the integral exists.
The inner product of two square-integrable functions $f,g$ is 
\begin{equation} \label{eqn:inner_product} 
\sprod{f,g} = \int_{\Xi} f(\xi) \, \conj{g(\xi)} \, \rho(\xi) \, \rmd \xi.
\end{equation}
In the following, $\ltwospace$ denotes the Hilbert space 
of square-integrable functions. 
The associated norm is $\norm{f}_{\ltwospace} = \sqrt{\sprod{f,f}}$.

Later we will focus on uniformly distributed random variables 
$\xi_{\ell} : \Omega \to \cc{-1,1}$. 
In this case, the joint probability density function is constant, i.e., 
$\Xi = [-1,1]^s$ and $\rho \equiv 2^{-s}$.

\subsection{Polynomial chaos expansions}
\label{sect:PCexpansion}

We assume that there is an orthonormal polynomial basis 
$(\phi_i)_{i \in \N_0}$ in $\ltwospace$. 
Thus it holds that
\begin{equation*}
\sprod{\phi_{i} , \phi_{j}} = \delta_{i,j}
=
\begin{cases}
1 & \text{for } i = j \\
0 & \text{for } i \neq j \\
\end{cases}
\end{equation*}
with the inner product~\eqref{eqn:inner_product}.
In the case of uniform probability distributions, 
the multivariate functions $\phi_i$ are products of the 
(univariate) Legendre polynomials. 
We assume that $\phi_0 \equiv 1$.
The number~$m+1$ of multivariate polynomials in $s$~variables 
up to a total degree~$r$ is
\begin{equation} \label{eqn:number_basis_polynomial}
m+1 = \frac{(s+r)!}{s! \, r!},
\end{equation}
see~\cite[p.~65]{Xiu2010}.
This number grows fast for increasing~$r$ or~$s$.

Let $u(x,\cdot) \in \ltwospace$ for each $x \in \overline{Q}$. 
The polynomial chaos (PC) expansion is
\begin{equation} \label{eqn:pc}
u(x, \xi) = \sum_{i=0}^\infty v_i(x) \phi_i(\xi)
\end{equation}
with (a priori unknown) coefficient functions 
\begin{equation} \label{eqn:coeff} 
v_i(x) = \sprod{u(x, \xi), \phi_i(\xi)} \quad \text{for } i \in \N_0.
\end{equation}
The series~\eqref{eqn:pc} converges in $\ltwospace$ 
pointwise for $x \in \overline{Q}$. 
If the wavenumber~$k$ is an analytic function of the random variables, 
then the rate of convergence is exponentially fast 
for traditional probability distributions.


\section{Discretization of the stochastic Helmholtz equation}
\label{sect:stochastic_Helmholtz}

We consider the stochastic Helmholtz equation
\begin{equation} \label{eqn:stochastic_helmholtz}
- \Delta u(x, \xi) - k(x, \xi)^2 u(x, \xi) = f(x), \quad x \in Q \subseteq \R^d,
\end{equation}
with given source term $f : Q \to \R$ and random wavenumber $k : 
\overline{Q} \times \Xi \to \R_+$, together with either homogeneous Dirichlet 
boundary conditions
\begin{equation} \label{eqn:stochastic_bc_dir}
u(x, \xi) = 0, \quad x \in \partial Q, \quad \xi \in \Xi,
\end{equation}
or with absorbing boundary conditions
\begin{equation} \label{eqn:stochastic_bc_abs}
\partial_n u(x, \xi) - \ii k(x, \xi) u(x, \xi) = 0, \quad x \in \partial Q, 
\quad \xi \in \Xi.
\end{equation}
All derivatives are taken with respect to $x$.
We discretize this boundary value problem in two steps, with a finite 
difference method (FDM) in space and the stochastic Galerkin method in the 
random-dependent part.
The steps can be done in any order.  We first give an overview of the procedure 
when beginning with the FDM in Section~\ref{sect:FDM_Galerkin}.
In Section~\ref{sect:Galerkin_FDM}, we discuss the discretization when 
beginning with the stochastic Galerkin method.

\subsection{FDM and stochastic Galerkin method}
\label{sect:FDM_Galerkin}

A spatial discretization of the boundary value problem with a FDM leads 
to a (stochastic) linear algebraic system
\begin{equation} \label{eqn:LGS_FD}
S(\xi) U(\xi) = F_0
\end{equation}
with $S(\xi) \in \K^{n,n}$ for $\xi \in \Xi$ (see 
Section~\ref{sect:discretization} for details) and constant vector $F_0 \in 
\R^n$.
In a second step, we consider a PC approximation of $U(\xi)$ of the form
\begin{equation} \label{eqn:Galerkin_approx}
\widetilde{U}_m(\xi) = \sum_{i=0}^m \phi_i(\xi) V_i, \quad
\text{where } V_i = \begin{bmatrix} v_{\ell, i} \end{bmatrix}_{\ell=1}^n \in 
\R^n \text{ for } i = 0, 1, \ldots, m,
\end{equation}
and $\phi_i$ are polynomials as in Section~\ref{sect:PCexpansion}.
The coefficient vectors $V_i$ are determined by the orthogonality of the 
residual
\begin{equation}
R_m(\xi) = S(\xi) \widetilde{U}_m(\xi) - F_0.
\end{equation}
to the subspace $\spann \{ \phi_0, \phi_1, \ldots, \phi_m \}$ 
with respect to the inner product 
$\sprod{\cdot, \cdot}$ in~\eqref{eqn:inner_product}, i.e., by
$\sprod{R_m(\xi), \phi_i(\xi)} = 0$ for $i = 0, 1, \ldots, m$.
Here the inner product is taken component-wise.  
The orthogonality condition is equivalent to
\begin{equation} \label{eqn:Galerkin_orthogonality}
\sprod{S(\xi) \widetilde{U}_m(\xi), \phi_i(\xi)} = \sprod{1, \phi_i(\xi)} F_0
= \delta_{i,0} F_0, \quad i = 0, 1, \ldots, m,
\end{equation}
due to $\phi_0 \equiv 1$.
This leads to a (deterministic) linear algebraic system
\begin{equation} \label{eqn:LGS_from_FD_G}
A V = F, \quad V = \begin{bmatrix} V_0 \\ \vdots \\ V_m \end{bmatrix}, \quad
F = \begin{bmatrix} F_0 \\ \vdots \\ F_m \end{bmatrix},
\end{equation}
where the stochastic Galerkin projection 
$A \in \K^{(m+1)n, (m+1)n}$ is a block 
matrix with $m+1$ blocks of size $n \times n$, 
and $F_i = 0 \in \R^n$ for $i = 1, \ldots, m$.

\begin{remark} \label{rem:galerkin_approx}
The Galerkin approximation~\eqref{eqn:Galerkin_approx} can be interpreted as a 
spatial discretization of a Galerkin approximation $\widetilde{u}_m(x, 
\xi) = \sum_{i=0}^m v_i(x) \phi_i(\xi)$ of $u(x, \xi)$.  Evaluating 
$\widetilde{u}_m$ at discretization points $x_1, \ldots, x_n$ yields
\begin{equation} \label{eqn:Galerkin_approx_discretized}
\begin{bmatrix} \widetilde{u}_m(x_1, \xi) \\ \vdots \\ \widetilde{u}_m(x_n, 
\xi) \end{bmatrix}
= \sum_{i=0}^m \phi_i(\xi) \begin{bmatrix} v_i(x_1) \\ \vdots \\ v_i(x_n) 
\end{bmatrix}.
\end{equation}
Hence $V_i$ in~\eqref{eqn:Galerkin_approx} can be interpreted as a
discretization of $v_i(x)$ by $v_{\ell,i} = v_i(x_\ell)$.
\end{remark}

As it turns out, the matrix $S(\xi)$ is a (complex) linear combination of 
symmetric positive (semi-)definite matrices.  The following lemma shows that 
this structure is preserved in the stochastic Galerkin method;
see~\cite[Lem.~1]{Pulch2019} and its proof.
These properties of the matrix $S$ and thus $A$ will be essential 
for our analysis of shifted Laplace preconditioners in Section~\ref{sect:csl}.

\begin{lemma} \label{lem:galerkin_projection}
Let $A(\xi) = \begin{bmatrix} a_{\mu, \nu}(\xi) \end{bmatrix}_{\mu,\nu} \in 
\R^{n,n}$ with $a_{\mu,\nu} \in \cL^2(\Xi, \rho)$, and $V \in \R^n$.
Define
\begin{equation}
A_{ij} \coloneq \begin{bmatrix} \sprod{a_{\mu, \nu}(\xi) \phi_i(\xi), 
\phi_j(\xi)} \end{bmatrix}_{\mu,\nu} \in \R^{n,n},
\quad i, j = 0, 1, \ldots, m,
\end{equation}
and the stochastic Galerkin projection
\begin{equation}
A \coloneq \begin{bmatrix} A_{ij} \end{bmatrix}_{i,j} \in \R^{(m+1)n, (m+1)n}.
\end{equation}
We then obtain for $i, j = 0, 1, \ldots, m$
\begin{equation}
\sprod{A(\xi) \phi_i(\xi) V, \phi_j(\xi)} = A_{ij} V,
\end{equation}
where the inner product is taken component-wise.
Additionally, $A_{ij} = A_{ji}$. 
Moreover, if $A(\xi)$ is symmetric, then $A$ is symmetric, 
and if $A(\xi)$ is symmetric positive (semi-)definite for almost all 
$\xi \in \Xi$ then $A$ is symmetric positive (semi-)definite.
\end{lemma}

\begin{corollary} \label{cor:constant_galerkin_projection}
In the notation of Lemma~\ref{lem:galerkin_projection}, if $A(\xi) = A_0$ is 
independent of $\xi$, then $A_{ij} = \delta_{ij} A_0$ and $A = I_{m+1} \otimes 
A_0$, with the identity matrix $I_{m+1} \in \K^{m+1, m+1}$ and the Kronecker 
product.
\end{corollary}

Finally, we obtain the following result on the structure of the matrix $A$ 
in~\eqref{eqn:LGS_from_FD_G}.

\begin{theorem} \label{thm:discretization_helmholtz}
Let the spatial dimension be $d \in \{ 1, 2 \}$.
A finite difference and stochastic Galerkin approximation of the Helmholtz 
equation~\eqref{eqn:stochastic_helmholtz} with either homogeneous Dirichlet or 
absorbing boundary conditions leads to a linear 
system~\eqref{eqn:LGS_from_FD_G} with coefficient matrix
\begin{equation} \label{eqn:A}
A = L - \ii B - K
\end{equation}
and real-valued matrices $L, B, K$.  
The matrix $K$ is symmetric positive definite, 
$B, L$ are symmetric positive semidefinite.
In case of homogeneous Dirichlet boundary conditions, 
$L$ is symmetric positive 
definite and $B = 0$.
\end{theorem}

\begin{proof}
The FD discretizations leading to $S(\xi) U(\xi) = F_0$ in~\eqref{eqn:LGS_FD} 
are given in Section~\ref{sect:discretization}.  The statement of the theorem 
follows in each case by applying the stochastic Galerkin approximation as 
described above and using Lemma~\ref{lem:galerkin_projection} as well as 
Corollary~\ref{cor:constant_galerkin_projection} separately for each term 
composing $S(\xi)$.
\end{proof}

The matrix $L$ results essentially from the discretization of the Laplacian, 
$B$ from the (absorbing) boundary conditions, 
and $K$ is the discretization of the term including the wavenumber; see 
Section~\ref{sect:discretization} for details.

\subsection{Stochastic Galerkin method and FDM}
\label{sect:Galerkin_FDM}

Alternatively, we can begin with the stochastic Galerkin method.  
This leads to a system of deterministic PDEs, 
which are subsequently discretized by a FDM.
The PC expansion~\eqref{eqn:pc} suggests a stochastic Galerkin approximation 
of $u(x, \xi)$ of the form
\begin{equation} \label{eqn:pc_galerkin}
\widetilde{u}_m(x, \xi) = \sum_{i=0}^m v_{i, m}(x) \phi_i(\xi).
\end{equation}
The coefficient functions $v_{i,m}$ in the stochastic Galerkin method 
are in general distinct from the coefficients $v_i$ in~\eqref{eqn:coeff}.  
Nevertheless, we will usually write $v_i$ instead of $v_{i, m}$ 
in the sequel for notational convenience.  
The coefficients in the Galerkin approach are determined by the 
orthogonality of the residual
\begin{align*}
R_m(x,\xi)
&= - \Delta \widetilde{u}_m(x, \xi) - k(x,\xi)^2 \widetilde{u}_m(x, \xi) 
- f(x) \\
&= - \sum_{i=0}^m \Delta v_i(x) \phi_i(\xi) 
- k(x,\xi)^2 \sum_{i=0}^m v_i(x) \phi_i(\xi) - f(x)
\end{align*}
to the subspace $\spann \{ \phi_0, \phi_1, \ldots, \phi_m \}$, i.e., by
$\sprod{ R_m(x,\xi) , \phi_j(\xi) } = 0$ for $j = 0,1,\ldots , m$ and each $x 
\in Q$.  The latter is equivalent to
\begin{equation} \label{eqn:pde_system_vj}
- \Delta v_j(x) - \sum_{i=0}^m \sprod{k(x,\xi)^2 \phi_i(\xi), \phi_j(\xi)} 
v_i(x) = \sprod{1, \phi_j(\xi)} f(x) = \delta_{j,0} f(x)
\end{equation}
for $j=0,1,\ldots,m$ in~$Q$.
Thus we obtain a system of PDEs for the unknown coefficient functions 
$v_0,v_1,\ldots,v_m$.
Define $C(x) = [c_{ij}(x)] \in \R^{m+1, m+1}$ for $x \in Q$ by
\begin{equation} \label{eqn:c_ij}
c_{ij}(x) = \sprod{k(x, \xi)^2 \phi_i(\xi), \phi_j(\xi)}
= \int_\Xi \phi_i(\xi) \phi_j(\xi) k(x, \xi)^2 \rho(\xi) \, \rmd \xi, 
\quad i, j = 0, 1, \ldots, m.
\end{equation}
Since by assumption $k(x,\xi) > 0$ for all $x$ and $\xi$, the matrix $C(x)$ is 
symmetric positive definite (as Gramian of an inner product with weight 
function $k(x, \xi)^2 \rho(\xi)$).
Setting
\begin{equation} \label{eqn:F}
v(x) = \begin{bmatrix} v_0(x) & v_1(x) & \cdots & v_m(x) \end{bmatrix}^\top, 
\quad
F(x) = \begin{bmatrix} f(x) & 0 & \cdots & 0 \end{bmatrix}^\top,
\end{equation}
we write the system of PDEs~\eqref{eqn:pde_system_vj} as 
\begin{equation} \label{eqn:de_v}
- \Delta v(x) - C(x) v(x) = F(x) \quad \text{in } Q,
\end{equation}
which is a larger deterministic system of linear PDEs.
Still we require boundary conditions for the system~\eqref{eqn:de_v}.

The homogeneous Dirichlet boundary condition~\eqref{eqn:stochastic_bc_dir} 
implies $v_j(x) = 0$ for $x \in \partial Q$ and $j = 0, 1, \ldots, m$, hence
\begin{equation}
v(x) = 0 \quad \text{on } \partial Q. \label{eqn:bc_v_dirichlet}
\end{equation}
Inserting the Galerkin approximation~\eqref{eqn:pc_galerkin} into the absorbing 
boundary conditions~\eqref{eqn:stochastic_bc_abs} yields the residual
\begin{align}
R_m(x,\xi)
&= \sum_{i=0}^m (\partial_n v_i)(x) \phi_i(\xi) - \ii k(x, \xi) \sum_{i=0}^m 
v_i(x) \phi_i(\xi).
\end{align}
By the orthogonality $\sprod{R_m(x, \xi), \phi_j(\xi)} = 0$ 
in the Galerkin approach, we obtain
\begin{equation} \label{eqn:bc_system_vj}
\partial_n v_j(x) - \ii \sum_{i=0}^m \sprod{k(x, \xi) \phi_i(\xi), \phi_j(\xi)} 
v_i(x) = 0, \quad j = 0, 1, \ldots, m.
\end{equation}
The matrix $B(x) = [b_{ij}(x)] \in \R^{m+1, m+1}$ with
\begin{equation} \label{eqn:b_ij}
b_{ij}(x) = \sprod{k(x, \xi) \phi_i(\xi), \phi_j(\xi)}
= \int_\Xi \phi_i(\xi) \phi_j(\xi) k(x, \xi) \rho(\xi) \, \rmd \xi,
\quad i, j = 0, 1, \ldots, m,
\end{equation}
is symmetric and positive definite (since $k(x, \xi) > 0$ by assumption).
The boundary condition~\eqref{eqn:bc_system_vj} can be written with $B(x)$ as
\begin{equation} \label{eqn:bc_v}
(\partial_n v)(x) - \ii B(x) v(x) = 0 \quad \text{on } \partial Q. 
\end{equation}
The boundary value problem~\eqref{eqn:de_v} with~\eqref{eqn:bc_v_dirichlet} 
or~\eqref{eqn:bc_v} is discretized in Section~\ref{sect:Galerkin_FDM_1d} (in 
dimension $d = 1$).  The resulting linear algebraic system is the same as the 
one obtained in Section~\ref{sect:FDM_Galerkin}.


\section{Complex shifted Laplace preconditioner}
\label{sect:csl}

Following the investigation in~\cite{GarciaSeteNabben2021}, we consider the 
Helmholtz equation~\eqref{eqn:stochastic_helmholtz} with a \emph{complex shift} 
in the wavenumber
\begin{equation} \label{eqn:shifted_helmholtz}
- \Delta u(x, \xi) - (1 + \ii \beta) k(x, \xi)^2 u(x, \xi) = f(x), \quad x \in 
Q,
\end{equation}
with $\beta \in \R$, together with either homogeneous Dirichlet boundary 
conditions~\eqref{eqn:stochastic_bc_dir} or absorbing boundary 
conditions~\eqref{eqn:stochastic_bc_abs}.
We discretize this boundary value problem as described in 
Section~\ref{sect:FDM_Galerkin}.  For $\beta = 0$, we have the 
matrix~\eqref{eqn:A} in Theorem~\ref{thm:discretization_helmholtz}, and for 
$\beta \in \R$ we obtain
\begin{equation} \label{eqn:M}
M \coloneq M(\beta) \coloneq L - \ii B - (1 + \ii \beta) K = A - \ii \beta K,
\end{equation}
since only the constant term is multiplied by $1 + \ii \beta$.
Motivated by~\cite[p.~1945]{vanGijzenErlanngaVuik2007}, we call $M$ a 
\emph{complex shifted Laplace preconditioner} (CSL preconditioner).

For the deterministic Helmholtz equation, preconditioning with the CSL 
preconditioner is a widely studied and successful technique for solving the 
discretized Helmholtz equation; see, e.g.,~\cite{ErlanggaVuikOosterlee2004, 
AiraksinenHeikkolaPennanenToivanen2007, RepsVanrooseZubair2010, 
CoolsVanroose2013,GanderGrahamSpence2015} and~\cite{GarciaSeteNabben2021}, as 
well as references therein.
See also~\cite{Erlangga2007} for a survey and~\cite{Lahaye2017} for recent 
developments.
In the deterministic case, 
the spectrum of the preconditioned matrix $A M^{-1}$ lies in the 
disk~\eqref{eqn:inclusion_disk}, and the improved 
localization of the spectrum typically leads to a faster convergence of Krylov 
solvers.  The CSL preconditioner $M$ can be inverted efficiently, 
for example, by multigrid techniques.

Here, we focus on locating the spectrum of the preconditioned matrix in the 
stochastic case, in analogy to~\cite{vanGijzenErlanngaVuik2007, 
GarciaNabben2018, GarciaSeteNabben2021} for the deterministic Helmholtz 
equation.

\begin{theorem} \label{thm:spectrum}
Let the notation be as in Theorem~\ref{thm:discretization_helmholtz}, let
$\beta > 0$, let $A$ be the discretization~\eqref{eqn:A} of the stochastic 
Helmholtz equation~\eqref{eqn:stochastic_helmholtz} and $M$ be the 
discretization~\eqref{eqn:M} of the shifted 
Helmholtz equation~\eqref{eqn:shifted_helmholtz}.
\begin{enumerate}
\item In the case of absorbing boundary 
conditions~\eqref{eqn:stochastic_bc_abs}, the spectrum of the preconditioned 
matrix $A M^{-1}$ is contained in the closed disk
\begin{equation} \label{eqn:inclusion_disk}
\mathcal{D} = \{ z \in \C : \abs{z - 1/2} \leq 1/2 \}.
\end{equation}
\item In the case of homogeneous Dirichlet boundary 
conditions~\eqref{eqn:stochastic_bc_dir}, the spectrum of the preconditioned 
matrix $A M^{-1}$ lies on the circle
\begin{equation} \label{eqn:circle}
\cC = \{ z \in \C : \abs{z - 1/2} = 1/2 \}.
\end{equation}
\end{enumerate}
\end{theorem}

\begin{proof}
We begin with the case of absorbing boundary conditions.
The proof closely follows~\cite[Sect.~3]{vanGijzenErlanngaVuik2007} with minor 
modifications.  We have
\begin{equation}
A = L - \ii B - z_1 K, \quad M = L - \ii B - z_2 K
\end{equation}
with $z_1 = 1$ and $z_2 = 1 + \ii \beta$ and where $L, B, K$ are symmetric, $K$ 
is positive definite and $L, B$ are positive semidefinite; see 
Theorem~\ref{thm:discretization_helmholtz}.
Then $A$ and $M$ are of the form in~\cite[Sect.~3]{vanGijzenErlanngaVuik2007}, 
except for the opposite sign of $B$.  The opposite sign affects the positive 
semidefiniteness, but not the overall strategy of the proof.  Nevertheless, we 
give a full proof here.

Step 1:
Observe first that $A M^{-1}$ and $M^{-1} A$ have the same spectrum, and that 
$M^{-1} A x = \sigma x$ is equivalent to the generalized eigenproblem $A x = 
\sigma M x$.

Step 2: $x$ is an eigenvector of $A x = \sigma M x$ if and only if $(L 
- \ii B) x = \lambda K x$, which can be seen as follows:
\begin{equation}
(L - \ii B - z_1 K) x = \sigma (L - \ii B - z_2 K) x
\Leftrightarrow (1 - \sigma) (L - \ii B) x = (z_1 - \sigma z_2) K x.
\end{equation}
For $\sigma \neq 1$, we obtain $(L - \ii B) x = \lambda K x$ with $\lambda = 
(z_1 - \sigma z_2)/(1 - \sigma)$.  (Note that $\sigma = 1$ is equivalent to 
$z_1 = z_2$, i.e., to $A = M$, which is excluded since $\beta > 0$.)
Conversely, if $(L - \ii B) x = \lambda K x$, then
$(L - \ii B - z_1 K) x = (\lambda - z_1) K x = \frac{\lambda - 
z_1}{\lambda - z_2} (L - \ii B - z_2 K) x$ and $\sigma = \frac{\lambda - 
z_1}{\lambda - z_2}$, provided that $\lambda \neq z_2$.
(Note that $(L - \ii B) x = z_2 K x$, i.e., $\lambda = z_2$, implies that $M$ 
is singular and thus not eligible as preconditioner.)

Step 3: Location of $\lambda$ in the generalized eigenvalue problem $(L - \ii 
B) x = \lambda K x$.
Since $K$ is symmetric positive definite, it has a Cholesky factorization $K = 
U U^\top$ and the generalized eigenvalue problem is equivalent to
\begin{equation} \label{eqn:transformed_eigenproblem}
U^{-1} (L - \ii B) U^{-\top} y = \lambda y,
\end{equation}
where $y = U^\top x$.
Multiplication of~\eqref{eqn:transformed_eigenproblem} by $y^\top$ and division 
by $y^\top y$ yields
\begin{equation}
\lambda = \frac{y^\top U^{-1} L U^{-\top} y}{y^\top y} - \ii \frac{y^\top 
U^{-1} B U^{-\top} y}{y^\top y}.
\end{equation}
This shows $\re(\lambda) \geq 0$ and $\im(\lambda) \leq 0$ since $L$ and $B$ 
are symmetric positive semidefinite.

Step 4: Estimate of the eigenvalues $\sigma$ of $M^{-1} A$. 
Since it holds that $z_1 \neq z_2$,
\begin{equation} \label{eqn:moebius}
\mu(z) = \frac{z - z_1}{z - z_2}
\end{equation}
is a M\"obius transformation.  By step~2, $\sigma = \mu(\lambda)$
where $\lambda$ is an eigenvalue of the generalized eigenvalue problem $(L - 
\ii B) x = \lambda K x$ which satisfies $\im(\lambda) \leq 0$.
To determine $\mu(\R)$, we compute
\begin{equation}
\mu(0) = \frac{z_1}{z_2} = \frac{1}{1 + \ii \beta} = \frac{1 - \ii \beta}{1 + 
\beta^2}, \quad \mu(z_1) = 0, \quad \mu(\infty) = 1
\end{equation}
and
\begin{equation}
\abs*{\mu(0) - \frac{1}{2}}^2
= \abs*{\frac{1}{1 + \beta^2} - \frac{1}{2} - \ii \frac{\beta}{1 + \beta^2}}
= \frac{(1 - \beta^2)^2}{4 (1 + \beta^2)^2} + \frac{\beta^2}{(1 + \beta^2)^2}
= \frac{1}{4}.
\end{equation}
Hence $\mu$ maps the real line onto the circle $\cC$ in \eqref{eqn:circle} 
(for any $\beta \neq 0$).  For $\beta > 0$, the lower half-plane 
is mapped by $\mu$ onto the interior of $\cC$ (for $\beta < 0$ onto the 
exterior); see Figure~\ref{fig:inclusion_region}.
This completes the proof in case of absorbing boundary conditions.

The proof in the case of Dirichlet boundary conditions is very similar.
The only difference is in the location of the eigenvalues $\lambda$ in step 3.
Since $L$ is symmetric positive definite and $B = 0$, 
\eqref{eqn:transformed_eigenproblem} implies $\lambda > 0$,
hence $\sigma = \mu(\lambda)$ lies on the circle~\eqref{eqn:circle}.
\end{proof}

\begin{figure}
{\centering
\includegraphics[width=0.5\linewidth]{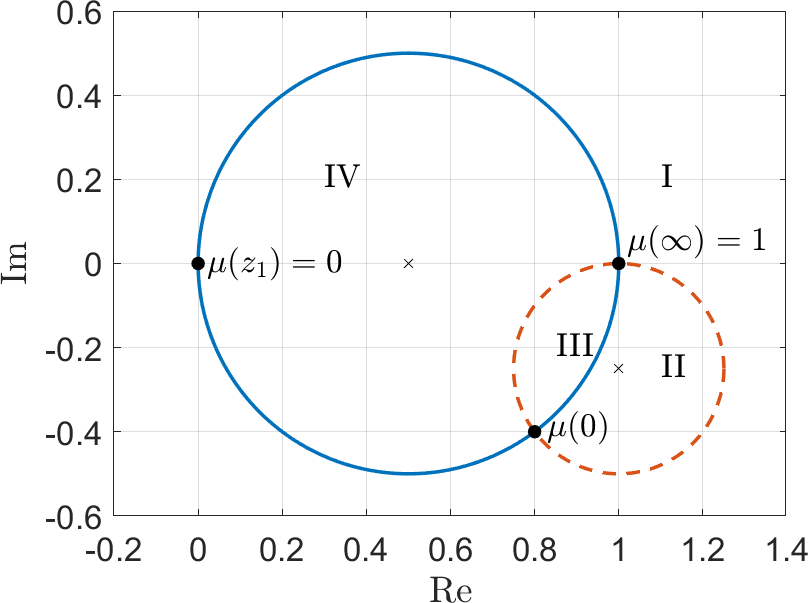}

}
\caption{Images under $\mu$ of the real and imaginary axis (solid and dashed 
circles, respectively) and of the four quadrants; see 
Remark~\ref{rem:inclusion_region} and the proof of Theorem~\ref{thm:spectrum}.}
\label{fig:inclusion_region}
\end{figure}

\begin{remark} \label{rem:inclusion_region}
In the proof of Theorem~\ref{thm:spectrum}, 
we additionally have $\re(\lambda) \geq 0$. 
Hence $\sigma$ is located in the image of the (closed) fourth quadrant 
under $\mu$ in~\eqref{eqn:moebius}.
To determine this image, 
note that $\mu$ maps the imaginary axis onto the circle
\begin{equation}
\cC_\beta = \{ z \in \C : \abs{z - (1 - \ii (\beta/2))} = 
\abs{\beta}/2 \},
\end{equation}
which intersects $\mu(\R) = \cC$ 
orthogonally in $\mu(0)$ and $\mu(\infty) = 1$.  Considering the orientations 
shows that $\mu$ maps the right half-plane onto the exterior of $\cC_\beta$; 
see Figure~\ref{fig:inclusion_region}.
Thus the spectrum satisfies
\begin{equation}
\sigma(A M^{-1}) \subseteq \{ z \in \C : \abs{z - 1/2} \leq 1/2 \}
\setminus \{ z \in \C : \abs{z - (1 - \ii (\beta/2))} < \abs{\beta}/2 \}.
\end{equation}
In case of Dirichlet boundary conditions, the eigenvalues of $A M^{-1}$ lie on 
the arc of the circle $\cC$ from $\mu(0)$ to $\mu(\infty) = 1$ that contains 
the origin.

This observation further tightens the inclusion set of $\sigma(A M^{-1})$, 
also in the case of a deterministic wavenumber.  This tighter inclusion set is 
already visible in~\cite[Fig.~1, Fig.~2]{GarciaNabben2018} 
and~\cite[Fig.~2.1]{GarciaSeteNabben2021} but we are not aware of a proof in 
the literature.
\end{remark}


\section{Mean value preconditioner}
\label{sect:mean-value}

We consider the discretization from Section~\ref{sect:FDM_Galerkin}.
Let $S(\xi) \in \K^{n,n}$ be the coefficient matrix of 
a linear system resulting 
from a spatial discretization of the Helmholtz equation~\eqref{eqn:de_orig} 
including boundary conditions and wavenumber $k(x,\xi)$. 
We assume that $S(\xi)$ is non-singular for almost all realizations 
$\xi \in \Xi$.
Let $\bar{\xi} \in \Xi$ be the expected value of the 
multidimensional random variable~$\xi$. 
It holds that
\begin{equation*}
S(\xi) = S(\bar{\xi}) + (S(\xi) - S(\bar{\xi})) \eqcolon
S(\bar{\xi}) + \Delta S (\xi).
\end{equation*}
The stochastic Galerkin method applied to $S(\xi)$ yields 
a matrix $A \in \mathbb{K}^{(m+1)n,(m+1)n}$ as shown 
in Section~\ref{sect:FDM_Galerkin}.
Furthermore, we define the constant matrix
\begin{equation} \label{eqn:bar_A}
\bar{A} = I_{m+1} \otimes S(\bar{\xi}).
\end{equation}
This matrix allows for the construction
\begin{equation} \label{eqn:Delta_A}
A = \bar{A} + (A-\bar{A}) \eqcolon \bar{A} + \Delta A.
\end{equation}
We employ the Frobenius matrix norm $\| \cdot \|_{\rm F}$ in the following.

\begin{theorem} \label{thm:matrix_bound}
Using the Frobenius norm, it holds that
\begin{equation} \label{eqn:matrix_bound1} 
\| \bar{A}^{-1} A - I_{(m+1)n} \|_{\rm F} \le 
C_m \, \| S(\bar{\xi})^{-1} \|_{\rm F} \,
\big\| \| \Delta S(\xi) \|_{\rm F} \big\|_{\ltwospace} 
\end{equation}
with the constants
\begin{equation*}
C_m = \sqrt{m+1} \,
\bigg( \sum_{i,j=0}^m \| \phi_i(\xi) \phi_j(\xi) \|_{\ltwospace}^2 
\bigg)^{\frac12}
\end{equation*}
provided that the $\mathcal{L}^2$-norm of the matrix norm is finite. 
\end{theorem}

\begin{proof}
The definition~\eqref{eqn:Delta_A} directly yields
\begin{equation*}
\bar{A}^{-1} A - I_{(m+1)n} = 
I_{(m+1)n} + \bar{A}^{-1} \Delta A - I_{(m+1)n} = 
\bar{A}^{-1} \Delta A.
\end{equation*}
We obtain 
$\| \bar{A}^{-1} \Delta A \|_{\rm F} \le 
\| \bar{A}^{-1} \|_{\rm F} \, \| \Delta A \|_{\rm F}$.  
The properties of the Kronecker product and~\eqref{eqn:bar_A} imply
$\| \bar{A}^{-1} \|_{\rm F}^2 = (m+1) \| S(\bar{\xi})^{-1} \|_{\rm F}^2$.  
We estimate $\norm{\Delta A}_{\rm F}$ using the Cauchy-Schwarz inequality with 
respect to the 
inner product~\eqref{eqn:inner_product}
\begin{align*} 
\| \Delta A \|_{\rm F}^2 & = 
\sum_{i,j=0}^{m} \sum_{\mu,\nu=1}^n 
\left| \sprod{ \phi_i(\xi) \phi_j(\xi) , \Delta S_{\mu,\nu}(\xi) } \right|^2 \\
& \le
\sum_{i,j=0}^{m} \sum_{\mu,\nu=1}^n 
\| \phi_i(\xi) \phi_j(\xi) \|_{\ltwospace}^2 \,
\| \Delta S_{\mu,\nu}(\xi) \|_{\ltwospace}^2 \\
& = 
\bigg( \sum_{i,j=0}^{m}  
\| \phi_i(\xi) \phi_j(\xi) \|_{\mathcal{L}^2(\Xi,\rho)}^2 \bigg) \,
\big\| \| \Delta S(\xi) \|_{\rm F} \big\|_{\ltwospace}^2 .
\end{align*}
In the last step, we used that the square of an $\mathcal{L}^2$-norm 
is an integral and thus summation (with respect to $\mu,\nu$) 
and integration can be interchanged. 
Applying the square root to the above estimate yields the 
statement~\eqref{eqn:matrix_bound1}. 
\end{proof}

\begin{remark} \label{rem:pessimistic_bound}
Rough estimates are used in the proof of Theorem~\ref{thm:matrix_bound}. 
Thus the true matrix norms of $\bar{A}^{-1} A - I_{(m+1)n}$ are often 
much smaller than the upper bounds in~\eqref{eqn:matrix_bound1}.
\end{remark}

\begin{remark}
If the random variable $\Delta S(\xi)$ is essentially bounded, 
then it follows that
\begin{equation*}
\big\| \| \Delta S(\xi) \|_{\rm F} \big\|_{\ltwospace} \le
\sup_{\xi \in \Xi \backslash \Upsilon} \| \Delta S(\xi) \|_{\rm F} < \infty
\end{equation*}
with a set $\Upsilon \subseteq \Xi$ of measure zero 
due to the normalization $\| 1 \|_{\ltwospace} = 1$.
\end{remark}

\begin{remark}
The bound of Theorem~\ref{thm:matrix_bound} also holds true for the 
Frobenius norm of $A \bar{A}^{-1} - I_{(m+1)n}$.
\end{remark}

Theorem~\ref{thm:matrix_bound} together with 
Remark~\ref{rem:pessimistic_bound} demonstrate that the matrix~$\bar{A}$ 
is a good preconditioner for solving linear systems with 
coefficient matrix~$A$. 
In this context, $\bar{A}$ is called the \emph{mean value preconditioner}, 
as in~\cite{Wang-etal} for the multi-element method.
When $\bar{A}$ is used as a preconditioner (left-hand or right-hand), 
linear systems with coefficient matrix $\bar{A}$ have to be solved. 
The matrix $\bar{A}$ from~\eqref{eqn:bar_A} is block-diagonal with 
$m+1$ identical blocks in this application. 
Thus just a single $LU$-decomposition of the matrix $S(\bar{\xi})$ 
is required. 
Many linear systems with different right-hand sides are solved using 
this $LU$-decomposition in an iterative method like GMRES, for example.

\begin{theorem} \label{thm:matrix_asymptotic}
Let $S(\xi) = S_0 + \theta T(\xi)$ with a non-singular constant matrix~$S_0$, 
a matrix $T= \begin{bmatrix} t_{\mu,\nu} \end{bmatrix}_{\mu,\nu}$ depending on 
a random variable~$\xi$ 
with components $t_{\mu,\nu}\in\ltwospace$
and a real parameter $\theta > 0$. 
Using $A_0 = I_{m+1} \otimes S_0$, 
the Frobenius norm exhibits the asymptotic behavior
\begin{equation} \label{eqn:matrix_behaviour} 
\| A_0^{-1} A - I_{(m+1)n} \|_{\rm F} = O(\theta) . 
\end{equation}
\end{theorem}

\begin{proof}
Since the entries of $T(\xi)$ are assumed to be square-integrable, also the 
expected values are finite.
Let $\bar{T}$
be the constant matrix containing the expected values of $T(\xi)$.  
We apply the decomposition
\begin{equation*}
S(\xi) = (S_0 + \theta \bar{T}) + \theta (T(\xi) - \bar{T}).
\end{equation*}
The matrix $S_0 + \theta \bar{T}$ is non-singular for sufficiently 
small~$\theta$.  Moreover, we obtain the relation
$(S_0 + \theta \bar{T})^{-1} = S_0^{-1} + O(\theta)$. 
Theorem~\ref{thm:matrix_bound} yields
\begin{equation*}
\| \bar{A}^{-1} A - I_{(m+1)n} \|_{\rm F} \le 
C_m \, \| (S_0 + \theta \bar{T})^{-1} \|_{\rm F} \,
\big\| \| \theta (T - \bar{T}) \|_{\rm F} \big\|_{\ltwospace}
\end{equation*}
with $\bar{A} = I_{m+1} \otimes (S_0+\theta \bar{T})$. 
It holds that $\bar{A} = A_0 + O(\theta)$ and thus $\bar{A}^{-1} = A_0^{-1} + 
O(\theta)$.  We conclude
\begin{equation*}
\| {A}_0^{-1} A - I_{(m+1)n} \|_{\rm F} \le 
\left( C_m \, \left( \| S_0^{-1} \|_{\rm F} + O(\theta) \right) \,
\theta \, \big\| \| T - \bar{T} \|_{\rm F} \big\|_{\ltwospace} \right) 
+ O(\theta)  = O(\theta),
\end{equation*}
which confirms~\eqref{eqn:matrix_behaviour}.
\end{proof}

An important case of Theorem~\ref{thm:matrix_asymptotic} is $\bar{T} = 0$, 
i.e., these expected values are zero. 
Then $A_0 = \bar{A}$ is the mean value preconditioner.

\begin{corollary} \label{cor:matrix_asymptotic} 
Under the assumptions of Theorem~\ref{thm:matrix_bound}, 
the Frobenius norm satisfies the estimate
\begin{equation} \label{eqn:matrix_norm_small} 
\| \bar{A}^{-1} A - I_{(m+1)n} \|_{\rm F} < 1 
\end{equation}
for all sufficiently small $\Delta S$.
\end{corollary}

Likewise, the Frobenius norm using $A_0$ instead of $\bar{A}$ is smaller 
than one if the parameter $\theta$ is sufficiently small 
in the context of Theorem~\ref{thm:matrix_asymptotic}.

A stationary iterative scheme for solving a linear system $A x = b$ 
reads as 
\begin{equation} \label{eqn:stationary_iteration}
B x^{(i+1)} = b - (A-B) x^{(i)} \qquad \mbox{for} \; i=0,1,2, \ldots 
\end{equation}
with a non-singular matrix~$B$ which should approximate~$A$, 
see~\cite[p.~621]{StoerBulirsch}.
In each iteration step, we have to solve a linear system with 
coefficient matrix~$B$. 
The property~\eqref{eqn:matrix_norm_small} is sufficient for the global 
convergence of the iteration~\eqref{eqn:stationary_iteration} 
using $B = \bar{A}$. 
The computational cost of an iteration step is much less than 
the steps in GMRES using~$\bar{A}$ as preconditioner, 
because the construction of Krylov subspaces is avoided.
In practice, we do not know if $\Delta S$ is sufficiently small 
such that the bound~\eqref{eqn:matrix_norm_small} is guaranteed. 
Nevertheless, it is worth to try this stationary iteration, 
as we will observe in Section~\ref{sect:experiments2d}.


\section{Numerical experiments in 1D}
\label{sect:experiments1d}

Our model problem in one space dimension is the stochastic Helmholtz 
equation~\eqref{eqn:stochastic_helmholtz} on $Q = \oo{0, 1}$ with absorbing 
boundary conditions.  The right-hand side is the point source 
$f(x) = \delta(x - \frac12)$, similarly to, 
e.g.,~\cite{GarciaSeteNabben2021, Livshits2015, 
SheikhLahayeGarciaNabbenVuik2016, vanGijzenErlanngaVuik2007}, 
where the right-hand side is a (possibly scaled) point source.
We consider a random wavenumber $k(x, \xi) = k(\xi)$ constant in space, 
which is uniformly distributed in some interval $\cc{k_{\min}, k_{\max}}$ 
with $0 < k_{\min} < k_{\max}$. 
Equivalently, we define
\begin{equation} \label{eqn:kp}
k(\xi) = (1 + \theta \xi) \overline{k}
\end{equation}
with a random variable $\xi$ that is uniformly distributed in $\cc{-1, 1}$, 
a mean value $\overline{k}$, and a real parameter~$\theta \in \oo{0, 1}$. 
It follows that $k_{\min} = (1-\theta)\overline{k}$
and $k_{\max} = (1+\theta)\overline{k}$.

In our numerical experiments in one and two spatial dimensions, we compute the 
mesh-size $h = \frac{1}{q + 1}$ in the FD discretization by
\begin{verbatim}
lev = max(ceil(log2((15*maxk)/(2*pi))), 1);
q = 2^lev - 1;
\end{verbatim}
where \verb|maxk| is the maximal value of the wavenumber.
Then the relation $\frac{2 \pi}{k h} \approx \text{constant}$, advocated 
in~\cite[Sect.~4.4.1]{Ihlenburg1998}, is satisfied.
Indeed, the estimate $x \leq \lceil x \rceil \leq x + 1$ for $x \in \R$ implies
$\frac{15k}{2 \pi} \leq q+1 \leq 2 \frac{15 k}{2 \pi}$ for large $k$.
In particular, $q$ grows linearly with $k$ and thus the size of the matrices 
$S(\xi)$ and $A$ (see Section~\ref{sect:discretization}) grows with $k$; see, 
e.g., Figure~\ref{fig:cond_A}.
Our choice for $q$ can be adapted for a future use of a multigrid method (as 
in~\cite{GarciaSeteNabben2021}).

Discretizing the model problem yields a linear algebraic system
\begin{equation} \label{eqn:unpreconditioned}
A x = b
\end{equation}
as given in Theorem~\ref{thm:discretization_1d_abs}.
This one-dimensional problem can be solved by a direct method, 
since the computational work is not too large.
Nevertheless we also consider its solution with the 
GMRES method~\cite{SaadSchultz1986} and investigate the application 
of CSL and mean value preconditioners introduced 
in Sections~\ref{sect:csl} and~\ref{sect:mean-value}, respectively.

By Theorem~\ref{thm:discretization_1d_abs}, the matrix $A$ has the form
\begin{equation}
A = I_{m+1} \otimes T - \ii [B_{ij}] - [C_{ij}].
\end{equation}
If needed, we write $A_\theta$ to indicate the dependence of $A$ on $\theta$, 
and in particular $A_0$ for $\theta = 0$, which corresponds to the 
mean value preconditioner.
Since the wavenumber in~\eqref{eqn:kp} is constant in space, the matrices 
$[B_{ij}]$ and $[C_{ij}]$ simplify to
\begin{align}
[B_{ij}] &= \begin{bmatrix} \sprod{k(\xi) \phi_j(\xi), \phi_i(\xi)} 
\end{bmatrix}_{ij} \otimes D_1,
& D_1 &= \frac{1}{h} \diag(1, 0, \ldots, 0, 1), \label{eqn:Bij_constant_k} \\
[C_{ij}] &= \begin{bmatrix} \sprod{k(\xi)^2 \phi_j(\xi), \phi_i(\xi)} 
\end{bmatrix}_{ij} \otimes D_2,
& D_2 &= \diag \Bigl( \frac{1}{2}, 1, \ldots, 1, \frac{1}{2} \Bigr),
\label{eqn:Cij_constant_k}
\end{align}
see Lemma~\ref{lem:discretization_1d_abs_k_xi}, and, by 
Lemma~\ref{lem:discretization_1d_abs_bandwith},
\begin{equation}
B_{ij} = 0 \quad \text{for } \abs{i-j} > 1, \quad
C_{ij} = 0 \quad \text{for } \abs{i-j} > 2.
\end{equation}
In other words, the matrices
$\begin{bmatrix} \sprod{k(\xi) \phi_j(\xi), \phi_i(\xi)} \end{bmatrix}_{ij}$
and
$\begin{bmatrix} \sprod{k(\xi)^2 \phi_j(\xi), \phi_i(\xi)} \end{bmatrix}_{ij}$ 
are tridiagonal and pentadiagonal, respectively, 
due to the orthogonality properties of the polynomials.

\begin{remark} \label{rem:block_diagonal}
In the deterministic case $k(\xi) = \overline{k}$ in~\eqref{eqn:kp}, 
i.e., $\theta = 0$, the matrices
$[B_{ij}] = \overline{k} I_{m+1} \otimes D_1$ and $[C_{ij}] = 
\overline{k}^2 I_{m+1} \otimes D_2$ are diagonal, and
\begin{equation}
A_0
= I_{m+1} \otimes (T - \ii \overline{k} D_1 - \overline{k}^2 D_2)
= I_{m+1} \otimes S(0)
\end{equation}
with $S(\xi)$ from Theorem~\ref{thm:discretization_1d_abs}.
This shows that $A_0$ is block-diagonal with $m+1$ identical diagonal blocks.  
The latter are the FD-discretization of the deterministic Helmholtz equation 
with wavenumber $\overline{k}$ (associated to $\xi = 0$).
\end{remark}

If not specified otherwise, we use $m = 3$ in the stochastic Galerkin method 
and $\theta = 0.1$ in~\eqref{eqn:kp}.
Finally, we also consider the shifted Helmholtz equation 
\eqref{eqn:shifted_helmholtz} with shift $\beta = \frac12$ and denote 
the CSL preconditioner by $M = M(\frac12)$, see~\eqref{eqn:M}.  
As for $A$, we write $M_\theta$ if we wish to emphasize the dependence on 
$\theta$.

The numerical experiments have been performed in the software package 
MATLAB R2020b on an i7-7500U @ 2.70GHz CPU with 16 GB RAM.

\subsection{Spectra}

By Theorem~\ref{thm:spectrum}, the eigenvalues of the CSL preconditioned matrix 
$A M^{-1}$ lie in the closed disk~\eqref{eqn:inclusion_disk}.
This is illustrated in the left panel of Figure~\ref{fig:spectrum}, which 
displays the spectra of $A M^{-1}$ (with $\theta = 0.1$) and $A_0 M_0^{-1}$ 
(i.e., with $\theta = 0$).
Each eigenvalue of $A_0 M_0^{-1}$ is $(m+1)$-fold, since $A_0 = I_{m+1} \otimes 
S(0)$ is block-diagonal with identical diagonal blocks, see 
Remark~\ref{rem:block_diagonal}, and similarly for $M_0$.
For $\theta \neq 0$, the matrix $A M^{-1}$ is not block-diagonal, and 
$A M^{-1}$ has clusters of $m+1$ eigenvalues close to each $(m+1)$-fold 
eigenvalue of $A_0 M_0^{-1}$.
This can be observed in the figure with $m + 1 = 4$.
The right panel in Figure~\ref{fig:spectrum} displays the spectrum 
of $A A_0^{-1}$ for the mean value preconditioner. 
The eigenvalues are clustered at $1$, 
which suggests a fast convergence of GMRES.
If the eigenvalues satisfy $\abs{\lambda - 1} < 1$ then the stationary 
method~\eqref{eqn:stationary_iteration} with $B = A_0$ converges.

\begin{figure}
{\centering
\includegraphics[width=0.47\linewidth]{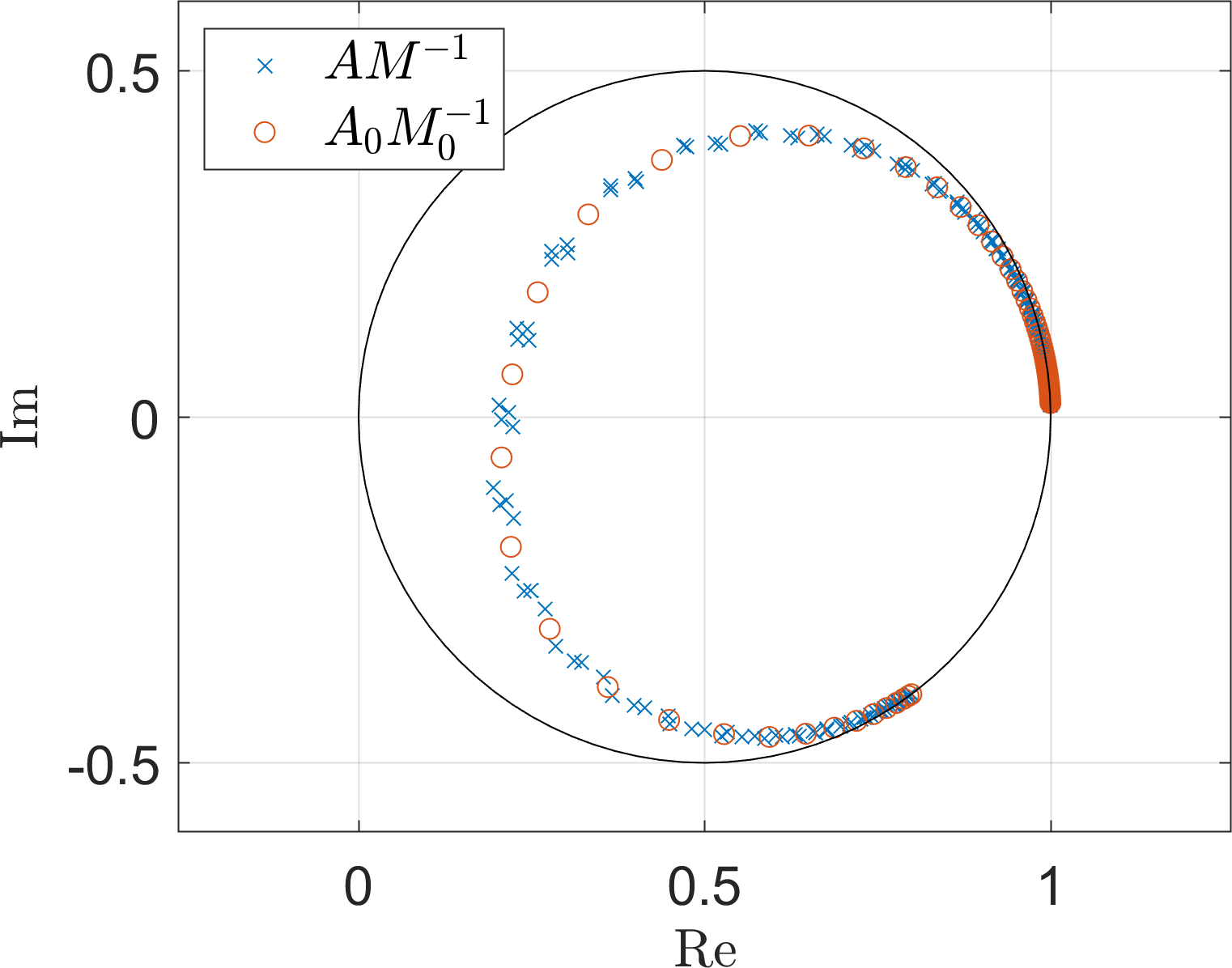}
\includegraphics[width=0.47\linewidth]{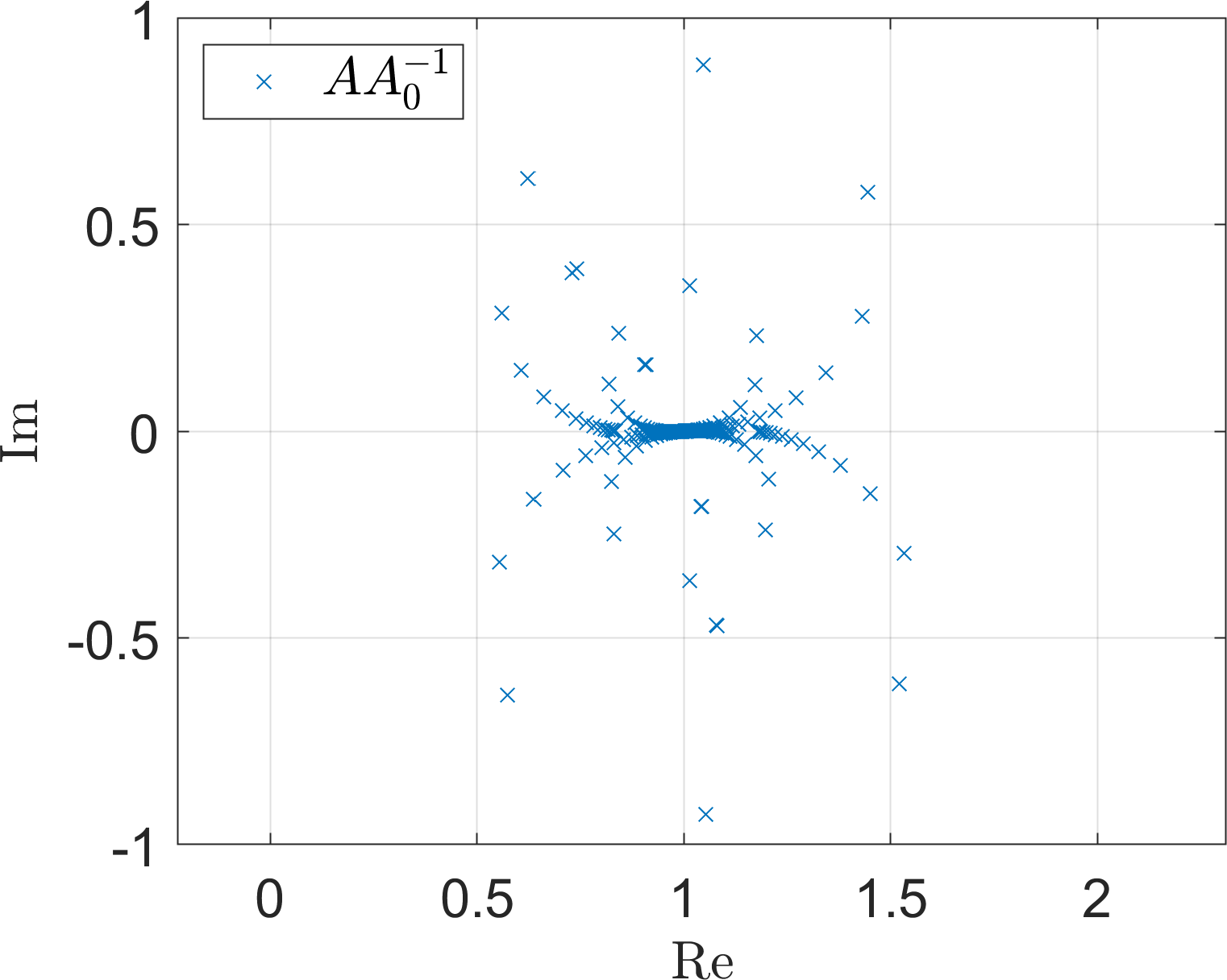}

}
\caption{Left: Spectrum of 
$A M^{-1}$ for $\overline{k} = 50$, $m = 3$, $\theta 
= 0.1$ (crosses) and $\theta = 0$ (circles).
The large solid circle illustrates~\eqref{eqn:circle}.
Right: Spectrum of $A A_0^{-1}$.}
\label{fig:spectrum}
\end{figure}

\subsection{Condition numbers}

Figure~\ref{fig:cond_A} displays the 2-norm condition numbers of $A$, $M$, $A 
M^{-1}$, $A_0$ and $A A_0^{-1}$ as functions of $\overline{k}$ (with $\theta = 
0.1$).  Clearly, the 
condition numbers of $M$ and $A M^{-1}$ are much smaller than the condition 
number of $A$, which is beneficial when solving the preconditioned linear 
system $A M^{-1} y = b$, $M x = y$ with the CSL preconditioner.
In this example, $\kappa_2(M) \leq 205$ for all $\overline{k}$, which is very 
moderate, and $\kappa_2(A M^{-1})$ grows linearly in $\overline{k}$ from 
$2.6485$ when $\overline{k} = 10$ to only $36.5190$ when $\overline{k} = 200$.
In contrast, $\kappa_2(A)$ is roughly $50$ to $160$ times larger than 
$\kappa_2(A M^{-1})$.
The observed spikes of $\kappa_2(A)$ occur when more discretization points are 
used which leads to a larger size of $A$, compare the curve of \verb|size(A)|.
The condition number of the mean value preconditioned matrix $A A_0^{-1}$ is 
also moderate, growing from $2$ to $141$, which is beneficial for 
solving the preconditioned linear system, while $\kappa_2(A_0)$ is of the order 
of $\kappa_2(A)$.

\begin{figure}
{\centering
\includegraphics[width=0.47\linewidth]{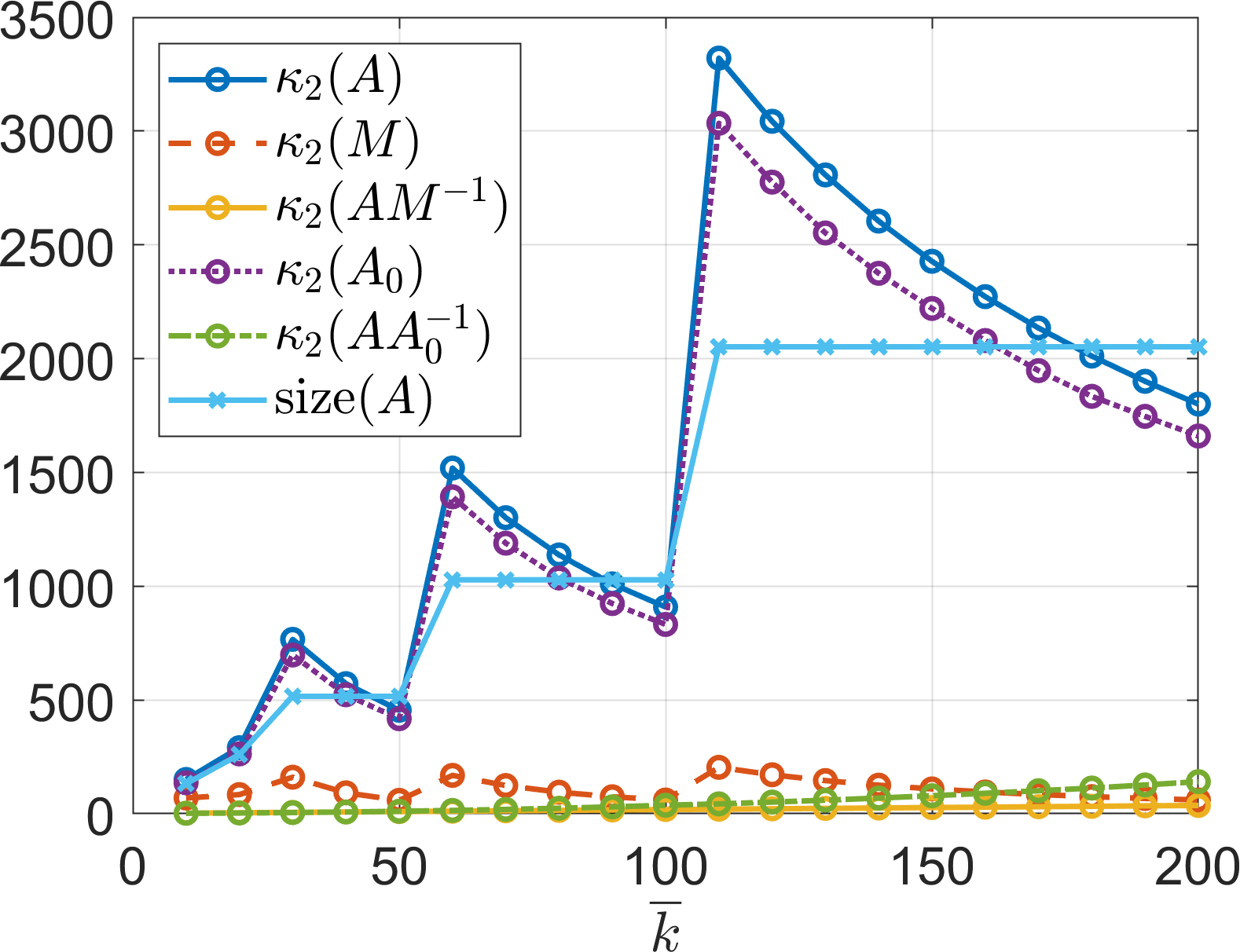}
\includegraphics[width=0.47\linewidth]{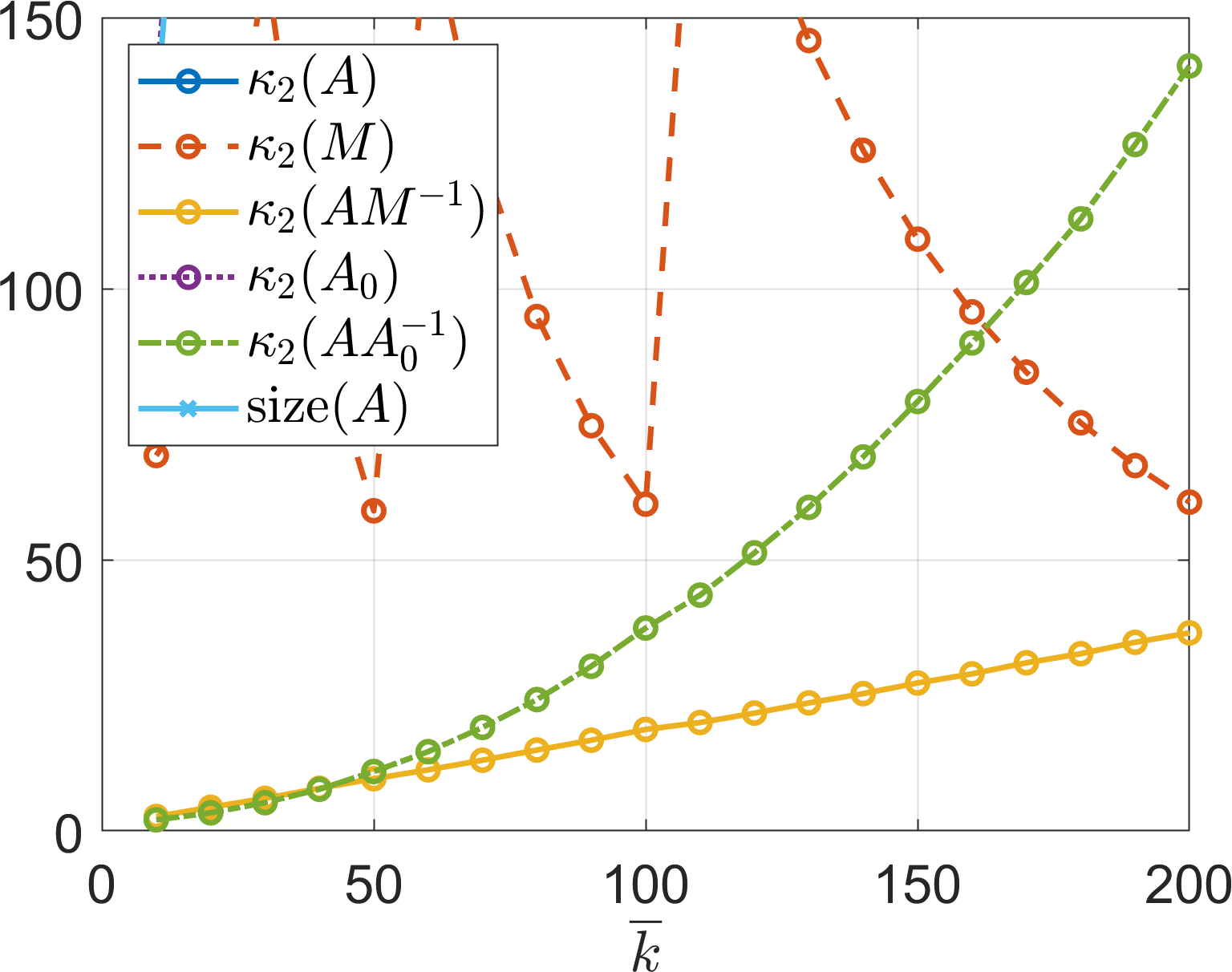}

}
\caption{$2$-norm condition numbers as functions of $\overline{k}$ (left) 
and zoom-in (right).}
\label{fig:cond_A}
\end{figure}

\subsection{GMRES}
\label{sect:gmres_1d}

We solve the unpreconditioned system~\eqref{eqn:unpreconditioned} and 
the right and left preconditioned systems
\begin{equation} \label{eqn:CSL_preconditioned}
A M^{-1} y = b, \quad x = M^{-1} y,
\qquad \text{and} \qquad
M^{-1} A x = M^{-1} b
\end{equation}
with full GMRES (no restarts) and tolerance \verb|tol=1e-12|, using MATLAB's 
built-in \verb|gmres| command.
The residual in the $i$th step is $r^{(i)} = b - A x^{(i)}$ 
for unpreconditioned and right preconditioned GMRES, 
and $M^{-1} r^{(i)}$ for left preconditioned GMRES.  
In particular, the stopping criterion for left and right preconditioning 
is in general different.
We will consider the following three preconditioners:
\begin{enumerate}
\item the CSL preconditioner $M$,
\item the mean value preconditioner $A_0$,
\item the mean value CSL preconditioner $M_0$.
\end{enumerate}
In preconditioned GMRES, we need to solve linear systems with the 
preconditioner, for which we use an $LU$-decomposition.
In one spatial dimension, this is not competitive with the direct solution (see 
the end of Section~\ref{sect:gmres_1d}), but in two spatial dimension the block 
structure of the preconditioners $A_0$ and $M_0$ leads to a competitive method.
In MATLAB, the $LU$-decomposition of the sparse matrix $M$ calls the associated 
routine from UMFPACK; see~\cite{davis}.  The decomposition has the form
\begin{equation}
P M Q = L U
\end{equation}
with a lower triangular matrix $L$, upper triangular matrix $U$, and two 
permutation matrices $P, Q$.  In our implementation, we use
\begin{verbatim}
[L, U, p, q] = lu(M, 'vector');
qt = []; qt(q) = 1:numel(q);
\end{verbatim}
where, instead of the matrices $P, Q$, only vectors $p, q$ representing the 
permutations are stored, and where the vector \verb|qt| describes the inverse 
mapping of the permutation defined by \verb|q|.
Then, we implement $M^{-1} x$ by
\begin{verbatim}
x = U\(L\x(p,:));
x = x(qt,:);
\end{verbatim} 
By Remark~\ref{rem:block_diagonal}, $A_0 = I_{m+1} \otimes S(0)$ is 
block-diagonal with equal diagonal blocks so that, for fixed $\overline{k}$, 
only a single $LU$-decomposition of $S(0) \in \K^{n,n}$ is necessary 
to compute $A_0^{-1} x$ for any vector $x \in \K^{(m+1)n}$.  In our 
implementation, we partition and reshape $x$ so that only one linear system 
with $S(0)$ is solved:
\begin{verbatim}
x = reshape(x, [n, m+1]);
x = U\(L\x(p,:));
x = x(qt,:);
x = reshape(x, [], 1);
\end{verbatim}
The preconditioner $M_0$ is implemented in the same way.

\begin{figure}
{\centering
\includegraphics[width=0.45\linewidth]{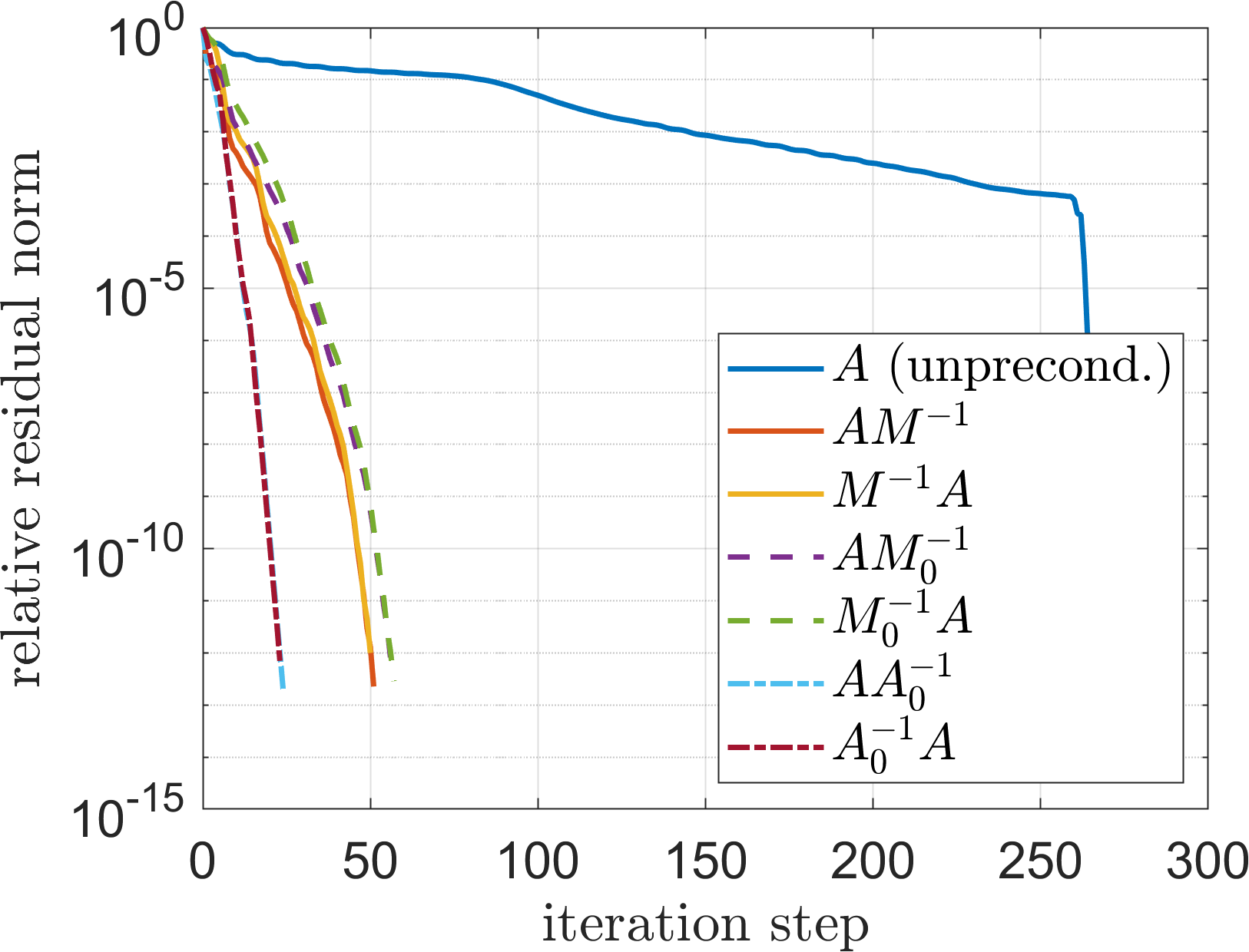}
\includegraphics[width=0.45\linewidth]{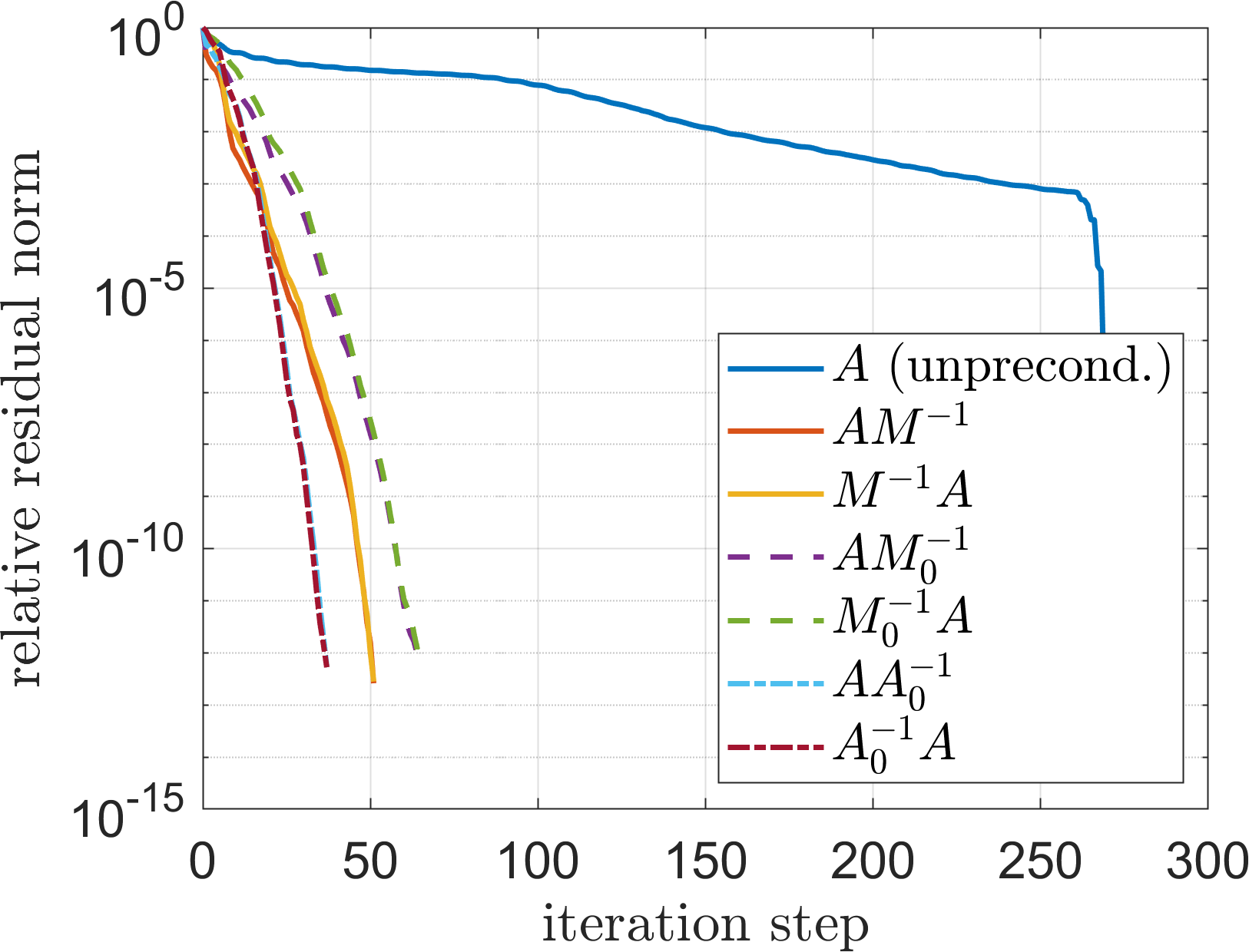}

}
\caption{Relative residual norms when solving~\eqref{eqn:unpreconditioned} 
with GMRES with various preconditioners, $m = 3$, $\overline{k} = 50$, and 
$\theta = 0.1$ (left) or $\theta = 0.2$ (right).}
\label{fig:gmres_relres}
\end{figure}

\begin{table}
{\centering
\begin{tabular}{ccccc}
\toprule
& \multicolumn{2}{c}{$\theta = 0.1$} & \multicolumn{2}{c}{$\theta = 0.2$} \\
preconditioner & left & right & left & right \\
\midrule
unpreconditioned & \multicolumn{2}{c}{$0.5509$} & \multicolumn{2}{c}{$0.5409$} 
\\
$M^{-1}$   & $0.0291$ & $0.0320$ & $0.0298$ & $0.0322$ \\
$M_0^{-1}$ & $0.0316$ & $0.0339$ & $0.0395$ & $0.0401$ \\
$A_0^{-1}$ & $0.0148$ & $0.0148$ & $0.0211$ & $0.0198$ \\
\bottomrule
\end{tabular}

}
\caption{Time in seconds (s) for solving~\eqref{eqn:unpreconditioned} 
and~\eqref{eqn:CSL_preconditioned} with GMRES.}
\label{tab:time_precond_kbar50}
\end{table}

In a first experiment, we fix $\overline{k} = 50$, $\theta = 0.1$ and $m = 3$.
Solving the unpreconditioned system~\eqref{eqn:unpreconditioned} with GMRES 
suffers from a long delay of convergence; see Figure~\ref{fig:gmres_relres}.
In contrast, all three preconditioners $M$, $M_0$, and $A_0$ lead to a 
significant decrease in the number of iteration steps from about $250$ to $50$ 
for $M$ and $M_0$ (factor~$5$), and to about $25$ for $A_0$ (factor $10$); see 
Figure~\ref{fig:gmres_relres} (left panel).
The computation times with the preconditioners $M$ and $M_0$ reduce to about 
$6\%$ of the computation time of unpreconditioned GMRES, while for $A_0$ it 
reduces to about $3\%$; see Table~\ref{tab:time_precond_kbar50}.
The computation times for the preconditioned systems include the computation of 
the $LU$-decomposition (of $M$ or of a diagonal block for $A_0$ or $M_0$).
The differences between computed solutions are very small: $\norm{x - 
x'}_\infty \leq 1.7 \cdot 10^{-14}$ (and typically of order $10^{-15}$), where 
\verb|x=A\b| is the direct solution and $x'$ is a solution computed with GMRES 
(unpreconditioned or with one of the preconditioners).
Left and right preconditioning lead to very similar relative residual norms and 
timings for each preconditioner.
A heuristic explanation why $A_0$ performs better than $M$ and $M_0$, is that 
$A$ is closer to $A_0$ than to $M$ or $M_0$.  Indeed, we have $\norm{A - 
A_0}_\infty < \norm{A - M}_\infty < \norm{A - M_0}_\infty$ in this example.
Repeating this experiment with $\theta = 0.2$ leads to very similar results, 
see Figure~\ref{fig:gmres_relres} and Table~\ref{tab:time_precond_kbar50}, 
so we focus on $\theta = 0.1$.

\begin{figure}[t!]
{\centering
\includegraphics[width=0.43\linewidth]{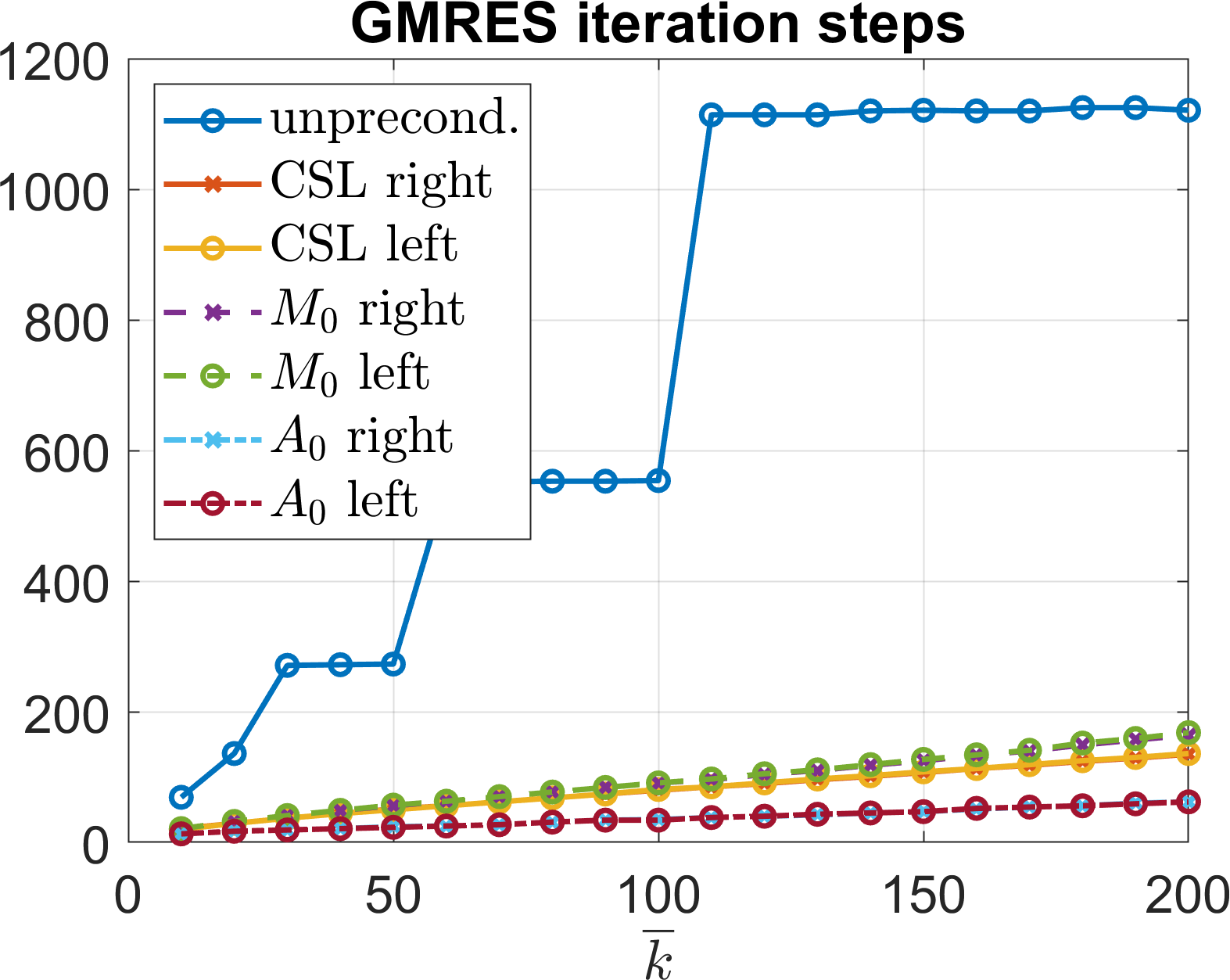}
\includegraphics[width=0.43\linewidth]
{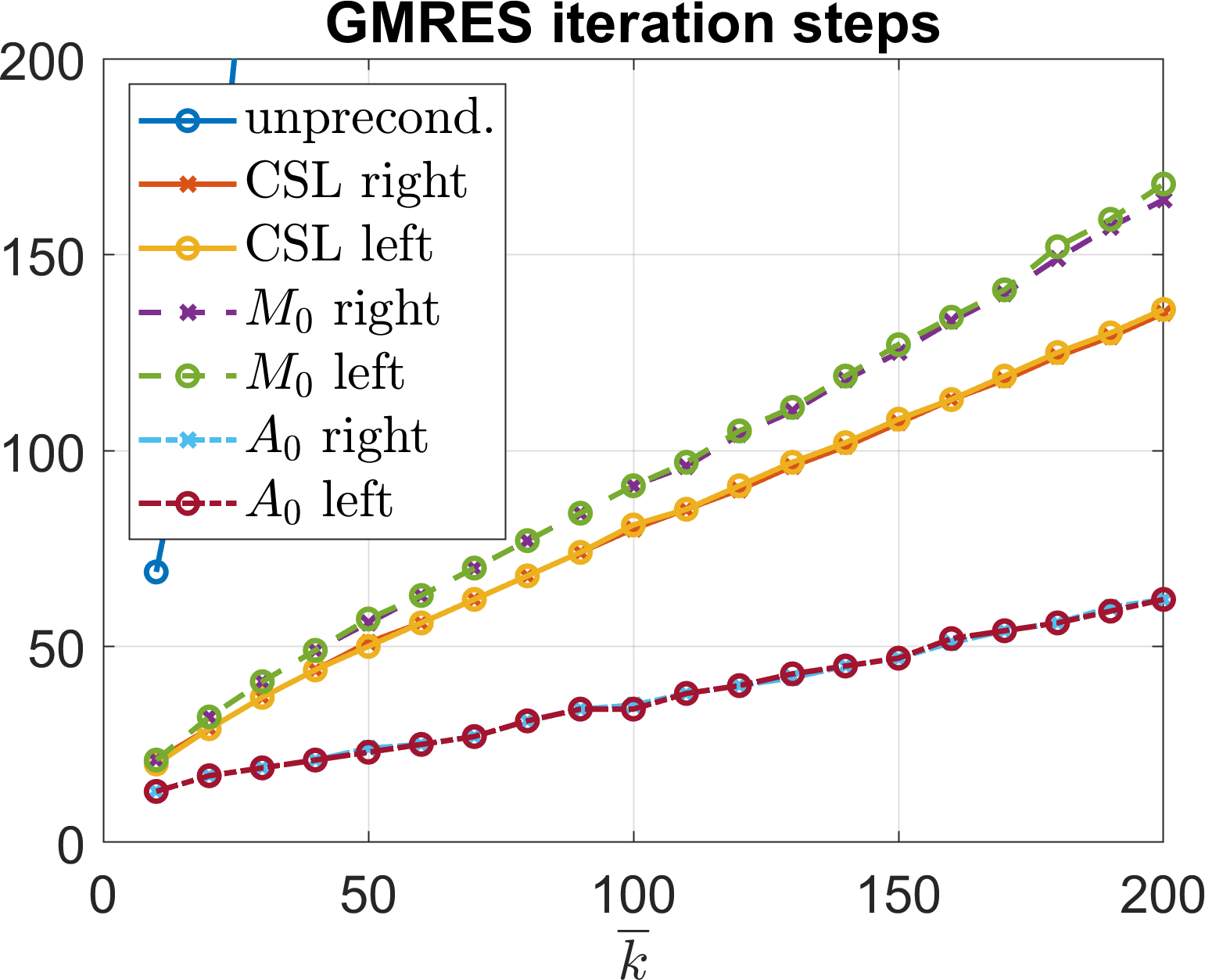}

\vspace{5mm}

\includegraphics[width=0.43\linewidth]{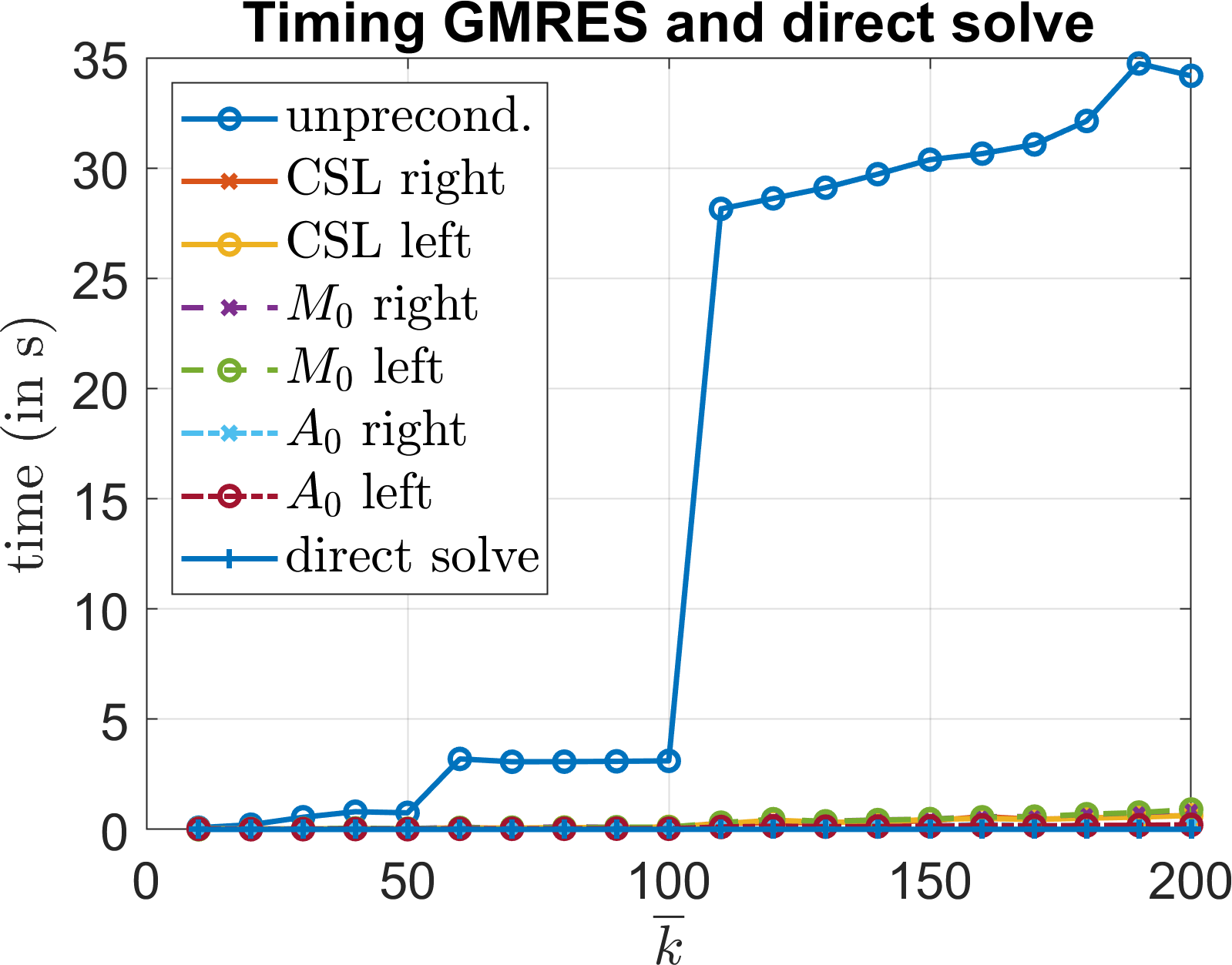}
\includegraphics[width=0.43\linewidth]
{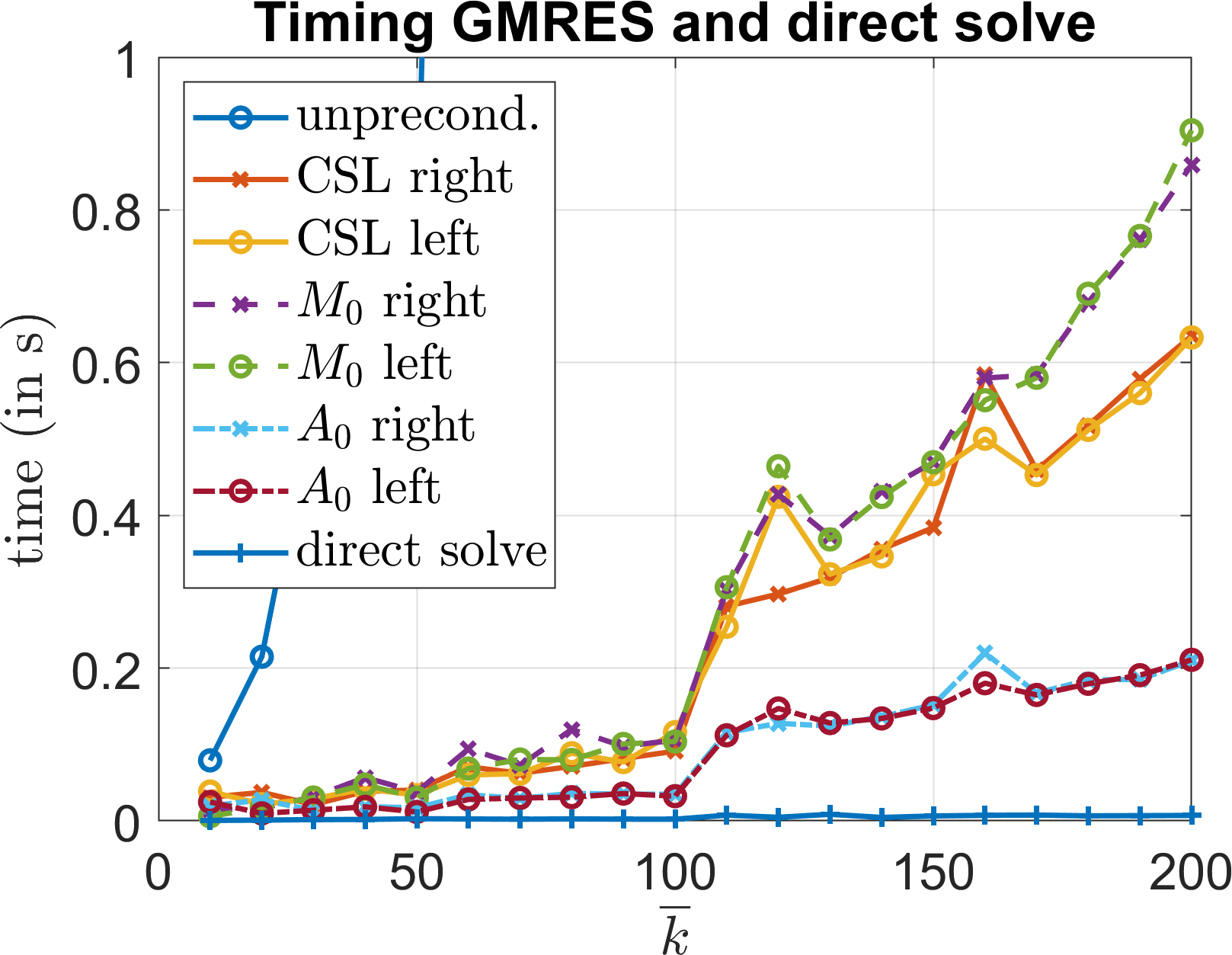}

}
\caption{Number of GMRES iteration steps (top) and computation time in seconds 
(bottom) as functions of $\overline{k}$ for different left and right 
preconditioners with fixed $\theta = 0.1$ and $m = 3$.  The right panels are 
zoom-ins.}
\label{fig:gmres_iter_vs_k}
\end{figure}

In a second experiment, we let $\overline{k}$ vary while $\theta = 0.1$ and $m 
= 3$ are fixed.  Figure~\ref{fig:gmres_iter_vs_k} displays the number of GMRES 
iteration steps (top) and the computation time (bottom) as functions of 
$\overline{k}$.
For small $\overline{k} \in \cc{10, 50}$, the difference between 
unpreconditioned and preconditioned GMRES is not so pronounced, since the 
linear systems are rather small.
For $60 \leq \overline{k} \leq 200$, the three preconditioners significantly 
reduce the number of iteration steps and the computation time compared to 
unpreconditioned GMRES.
The number of iteration steps is reduced to 8--15\% of the number of 
iteration steps in unpreconditioned GMRES when using $M$, to 9--16\% when 
using $M_0$ and to only 3--6\% when using $A_0$ as preconditioner.
GMRES preconditioned with $M$ or $M_0$ needs only 1--4\% of the computation 
of unpreconditioned GMRES, and the computation time of GMRES preconditioned 
with $A_0$ is reduced to 0.5--1.1\% of the computation time of 
unpreconditoned GMRES.
The mean value preconditioner $A_0$ leads to the smallest number of GMRES 
iteration steps and computation time, which is likely due to the fact that 
$A$ is closer to $A_0$ than to $M$ or $M_0$.
Note, however, that the condition number of $A_0$ (and $A$) is much 
larger than that of $M$ and $M_0$.  For $\overline{k} = 150$, we have (rounded 
to the nearest integer)
$\kappa_2(A) = 2428$, $\kappa_2(A_0) = 2220$, $\kappa_2(M) = 109$, 
$\kappa_2(M_0) = 91$;
see also Figure~\ref{fig:cond_A}.
Thus, if accuracy is an issue, it is preferable to work with the CSL 
preconditioners $M$ or $M_0$.

Finally, we note that the direct solution \verb|A\b| with a 
sparse matrix in MATLAB calls an efficient algorithm from UMFPACK; 
see~\cite{davis}.
In the above test example, solving the linear system
\eqref{eqn:unpreconditioned} by GMRES (with or without preconditioner) 
is not competitive with this direct solution, as it is much faster; 
see the bottom right panel in Figure~\ref{fig:gmres_iter_vs_k}.

\subsection{Solutions}

Figure~\ref{fig:gmres_solutions} displays the real and imaginary parts of the 
computed coefficients $v_0$, $v_1$, $v_2$, $v_3$ in the Galerkin approximation 
for $\overline{k}=50$ and $\theta = 0.1$ in~\eqref{eqn:kp}. 
We recognize an effect of the point source at $x=\frac12$ in 
the real part of~$v_0$.

\begin{figure}[t!]
{\centering
\includegraphics[width=0.45\linewidth]{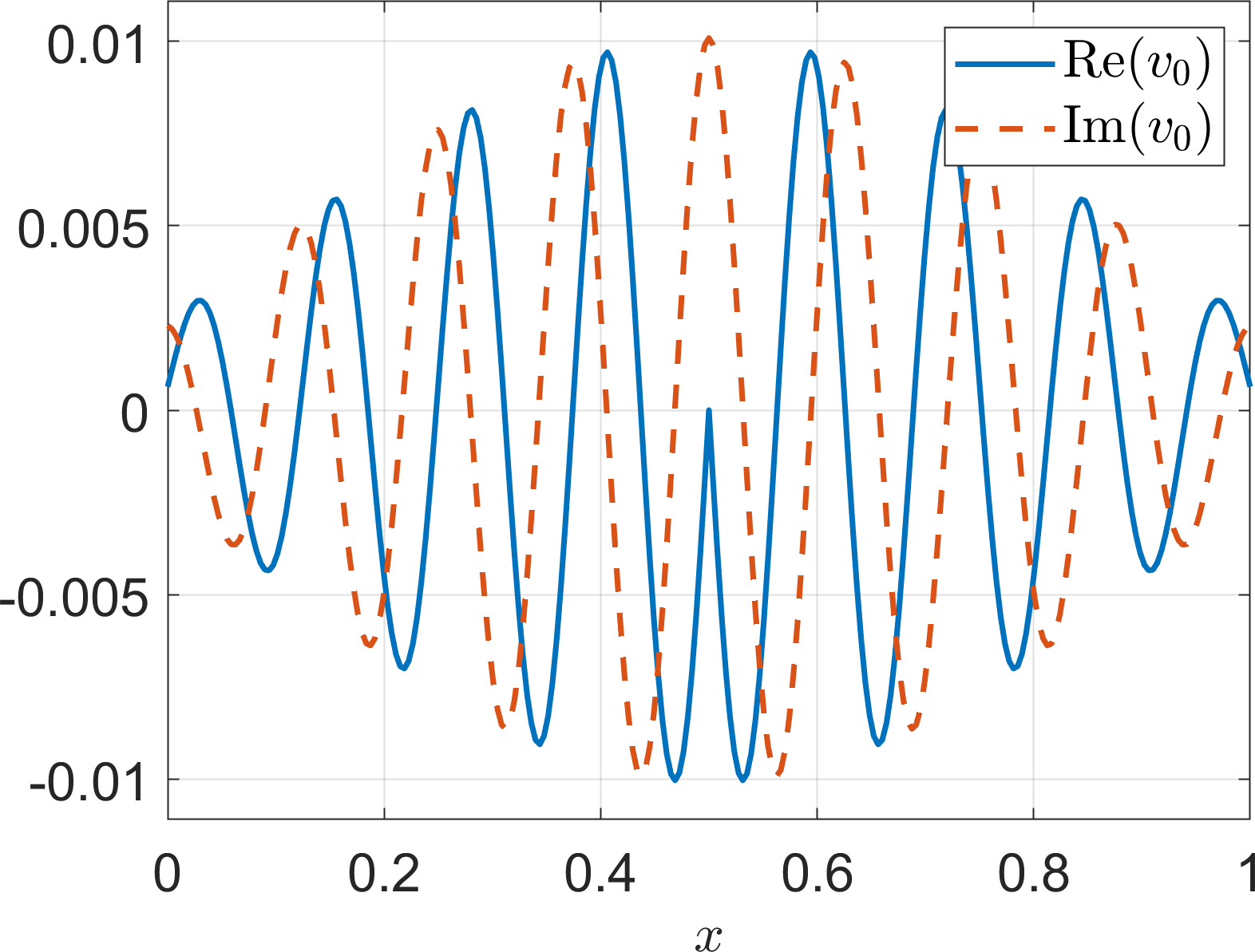}
\includegraphics[width=0.45\linewidth]{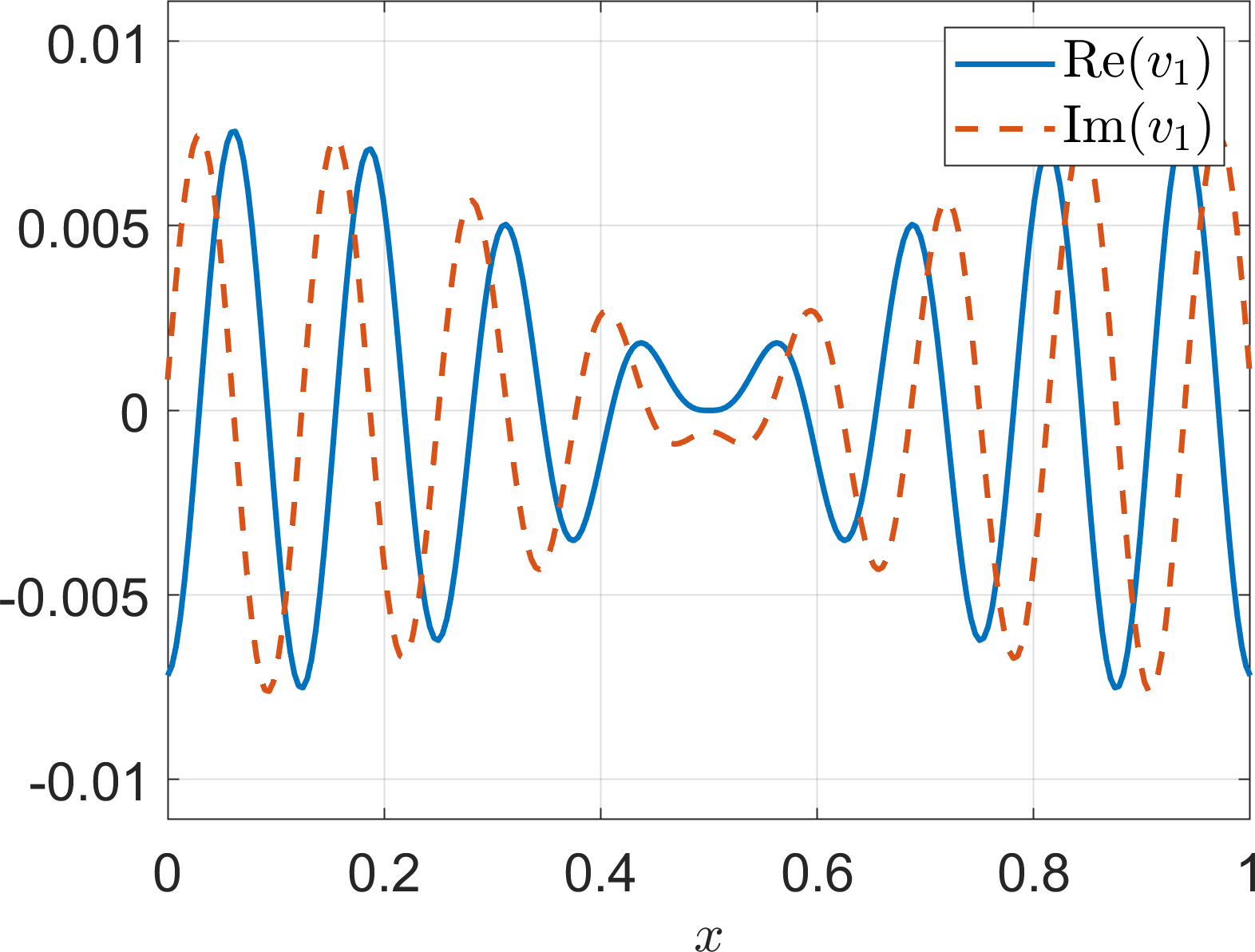}
\vspace{2mm}

\includegraphics[width=0.45\linewidth]{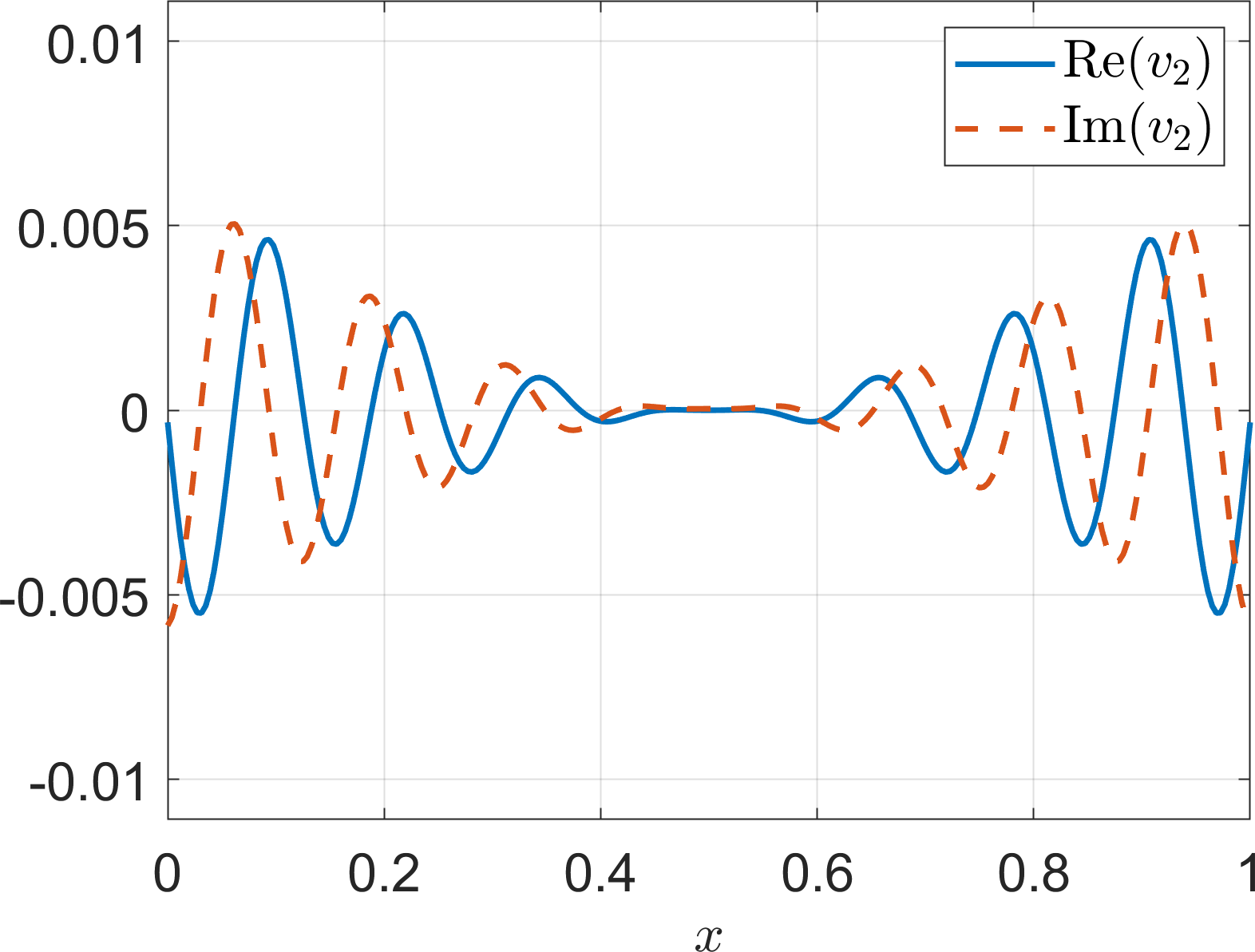}
\includegraphics[width=0.45\linewidth]{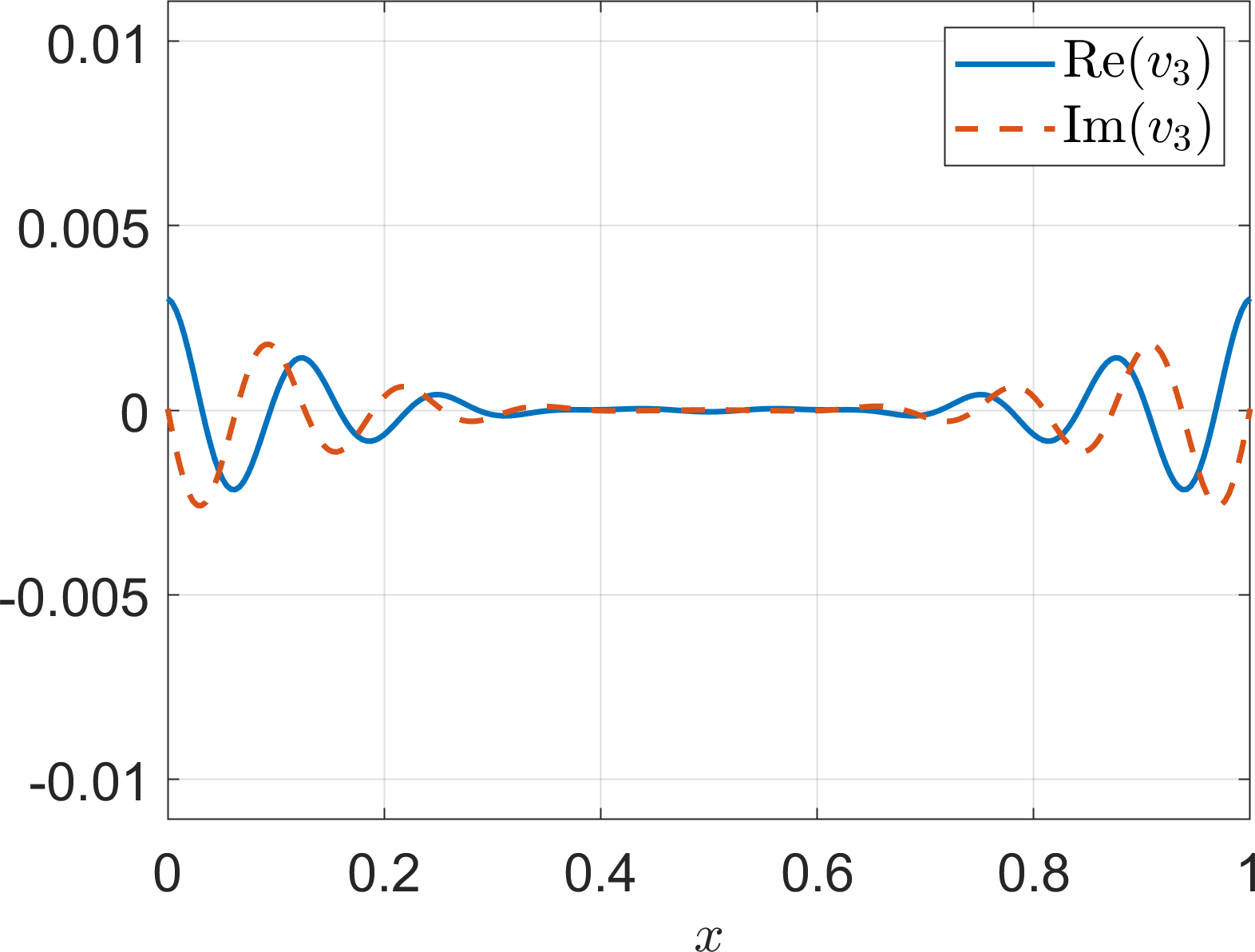}

}
\caption{Plots of the coefficients $v_0, v_1, v_2, v_3$ for 
$\overline{k} = 50$ and $\theta = 0.1$.}
\label{fig:gmres_solutions}
\end{figure}

\begin{figure}[t!]
{\centering
\includegraphics[width=0.47\linewidth]{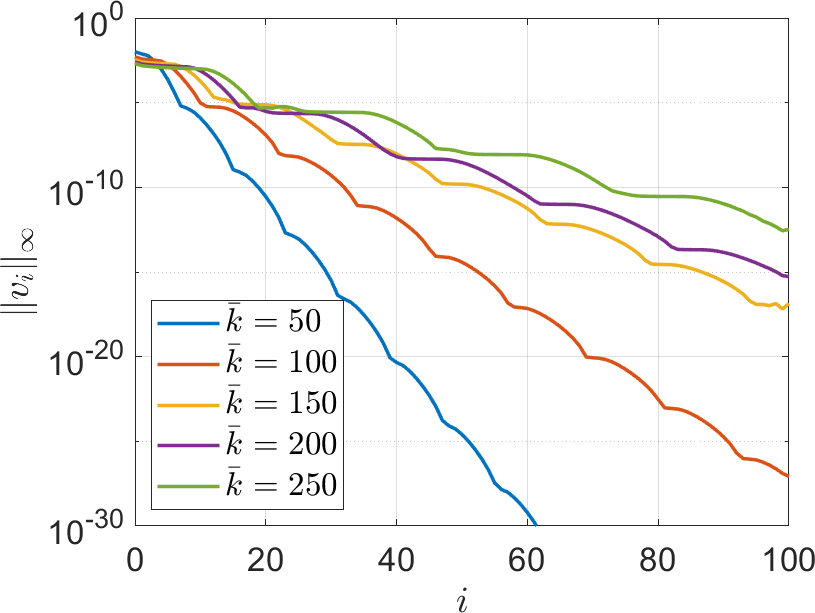}
\includegraphics[width=0.47\linewidth]{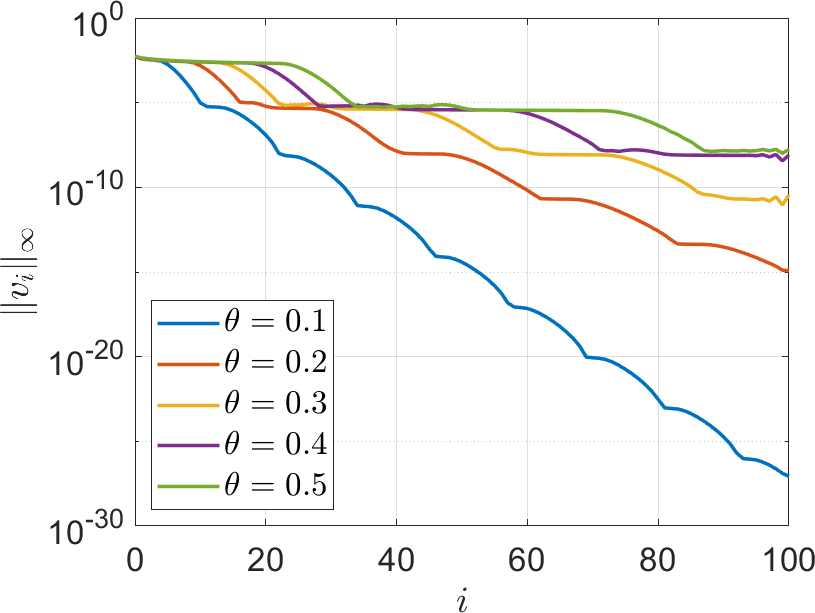}

}
\caption{Maximum-norms $\norm{v_i}_\infty$ as a function of $i = \deg(\phi_i)$.
Left: For fixed $\theta = 0.1$ and different values of $\overline{k}$.
Right: For fixed $\overline{k} = 100$ and different values of $\theta$.}
\label{fig:supnorm_1d}
\end{figure}

We compute the solution for total polynomial degree $m = 100$.
Figure~\ref{fig:supnorm_1d} shows $\norm{v_i}_\infty$ as a function of the 
polynomial degree $i$.
In the left panel, $\theta = 0.1$ is fixed and $\overline{k}$ varies, while in 
the right panel $\theta$ varies and $\overline{k} = 100$ is fixed.
We observe an exponential decay of the coefficients in all cases, which is 
related to the exponential convergence of the PC expansion~\eqref{eqn:pc}.
Larger wavenumbers and larger values of $\theta$ lead to a slower decay of the 
maximum-norm of the coefficients.
The effect of larger $\theta$ on the convergence/decay is more pronounced,
compare, for example, 
the curve for $(\overline{k}, \theta) = (150, 0.1)$ in the left 
panel with the curve for $(100, 0.5)$ in the right panel.

\begin{figure}[t!]
{\centering
\includegraphics[width=0.47\linewidth]{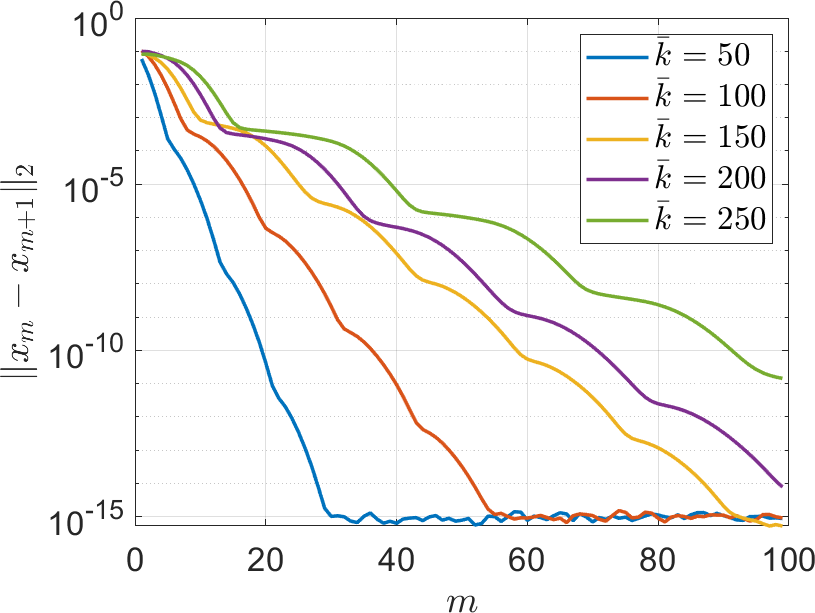}
\includegraphics[width=0.47\linewidth]{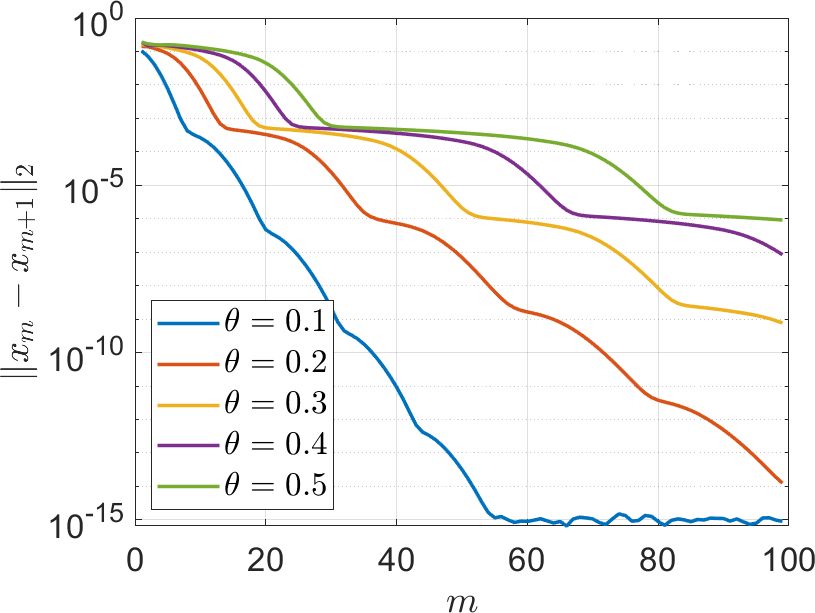}

}
\caption{Norms $\norm{x_m - x_{m+1}}_2$ as a function of the maximal degree 
$m$ in the stochastic Galerkin method.
Left: For fixed $\theta = 0.1$ and different values of $\overline{k}$.
Right: For fixed $\overline{k} = 100$ and different values of $\theta$.}
\label{fig:l2norm_1d}
\end{figure}

Next, we vary $m$ (the maximal degree of the polynomials in the stochastic 
Galerkin method) and denote by $x_m$ the solution 
of~\eqref{eqn:unpreconditioned}, which consists of a discretization of 
the coefficients $v_{0,m}, \ldots, v_{m,m}$ in a Galerkin 
approximation~\eqref{eqn:pc_galerkin} of the solution $u$ of the Helmholtz 
equation; see also Remark~\ref{rem:galerkin_approx}.
The convergence of the stochastic Galerkin method is illustrated by the 
exponential decay of the norms $\norm{x_m - x_{m+1}}_2$ in 
Figure~\ref{fig:l2norm_1d}.


\section{Numerical experiments in 2D}
\label{sect:experiments2d}

We consider the stochastic Helmholtz equation~\eqref{eqn:stochastic_helmholtz} 
in $Q = \oo{0, 1}^2$ with absorbing boundary 
conditions~\eqref{eqn:stochastic_bc_abs},
the point source $f(x,y) = \delta( (x,y) - (\frac{1}{2}, \frac{1}{2}))$ as 
right-hand side,
and space-dependent random wavenumber
\begin{equation} \label{eqn:stochastic_wavenumber_k123}
k(x, y, \xi_1, \xi_2, \xi_3) =
\begin{cases}
(1 + \theta \xi_1) k_1, & y \leq 0.2 + 0.1 x, \\
(1 + \theta \xi_2) k_2, & 0.2 + 0.1 x < y < 0.6 -0.2 x, \\
(1 + \theta \xi_3) k_3, & 0.6 -0.2 x \leq y.
\end{cases}
\end{equation}
on the wedge-shaped domain from~\cite[p.~146]{Livshits2015}; 
similar domains have been examined in 
\cite[Sect.~6.3]{ErlanggaVuikOosterlee2004} 
and~\cite[Sect.~4.4]{GarciaNabben2018}.
The modeling~\eqref{eqn:stochastic_wavenumber_k123} can also be written 
in the form~\eqref{eqn:random_wavenumber} using spatial indicator functions.
The random variables $\xi_1, \xi_2, \xi_3$ are independent and uniformly 
distributed in $\cc{-1, 1}$.  
The mean value of the wavenumber is
\begin{equation} \label{eqn:wavenumber_k123}
\overline{k}(x, y) =
\begin{cases}
k_1, & y \leq 0.2 + 0.1 x, \\
k_2, & 0.2 + 0.1 x < y < 0.6 - 0.2 x, \\
k_3, & 0.6 - 0.2 x \leq y.
\end{cases}
\end{equation}
We discretize the boundary value problem as described in 
Section~\ref{sect:FDM_Galerkin} and obtain the linear algebraic system
$A x = b$ in Theorem~\ref{thm:discretization_2d_abs}.
The number of polynomials in the three random variables $\xi_1, \xi_2, \xi_3$ 
with total degree at most $r$ is, 
see~\eqref{eqn:number_basis_polynomial},
\begin{equation} \label{eqn:number_p_ex}
m + 1 = \textstyle \frac{(r+3)!}{r! \, 3!} .
\end{equation}
Table~\ref{tab:construct_A} includes the number of basis polynomials for 
degrees $r = 0, 1, \ldots, 8$.

\begin{table}
{\centering
\begin{tabular}{crrrr}
\toprule
$r$ & n.basis & size of $A$ & nnz & time (s) \\
\midrule
$0$ & $1$ & $16641$ & $82689$ & $0.3264$ \\
$1$ & $4$ & $66564$ & $364038$ & $0.6592$ \\
$2$ & $10$ & $166410$ & $993300$ & $1.4090$ \\
$3$ & $20$ & $332820$ & $2119728$ & $3.3421$ \\
$4$ & $35$ & $582435$ & $3892575$ & $10.3049$ \\
$5$ & $56$ & $931896$ & $6461094$ & $32.8872$ \\
$6$ & $84$ & $1397844$ & $9974538$ & $101.1674$ \\
$7$ & $120$ & $1996920$ & $14582160$ & $291.5063$ \\
$8$ & $165$ & $2745765$ & $20433213$ & $802.0748$ \\
\bottomrule
\end{tabular}

}
\caption{For different total degrees $r$, number of basis polynomials 
(n.basis), 
size of the matrix~$A$, number of non-zero elements in $A$ (nnz), and time to 
generate $A$.}
\label{tab:construct_A}
\end{table}

Let $k_1 = 30$, $k_2 = 15$, $k_3 = 20$, and $\theta = 0.1$.
Table~\ref{tab:construct_A} shows the size of $A$ and the time (in seconds) 
for constructing the matrix $A$ for polynomial degrees up to 
$r = 0, 1, \ldots, 8$ in the stochastic Galerkin method.
The computation time when solving $A x = b$ directly with \verb|x=A\b| in 
MATLAB grows exponentially as a function of~$r$, 
see Figure~\ref{fig:2d_dir_vs_gmres} (left panel).

\begin{figure}[t!]
{\centering
\includegraphics[width=0.47\linewidth]{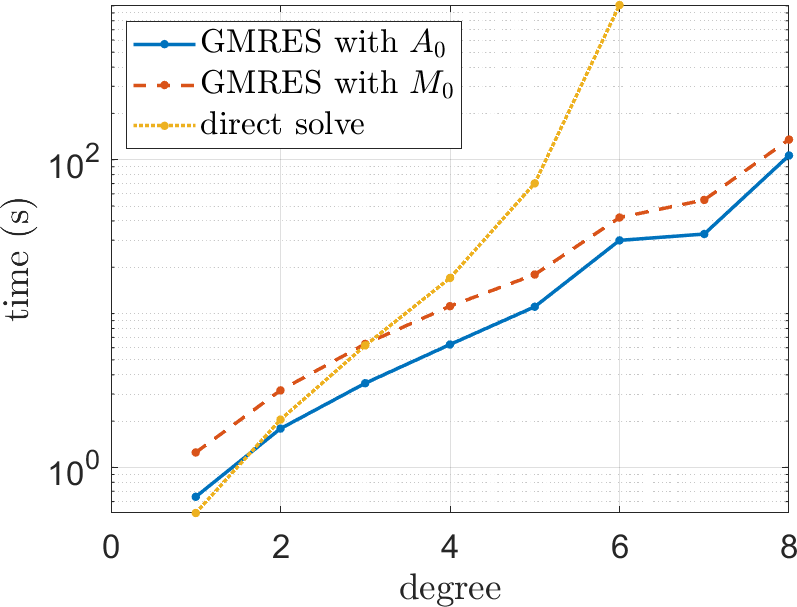}
\includegraphics[width=0.47\linewidth]{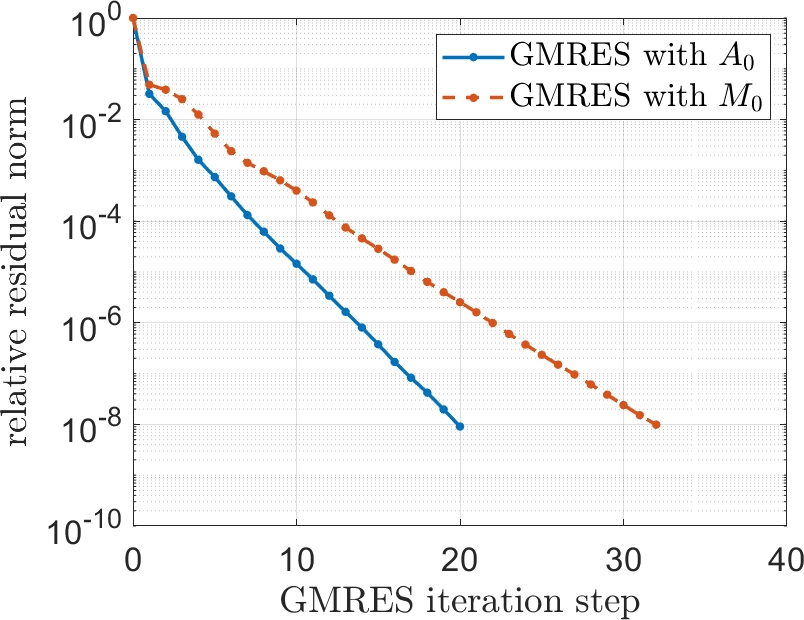}

}
\caption{Solving the linear system directly and with GMRES preconditioned 
by $A_0$ and $M_0$; see Section~\ref{sect:experiments2d}.
Left: Computation time (in seconds) as a function of the polynomial degree~$r$ 
in the stochastic Galerkin method.
Right: Relative residual norms in preconditioned GMRES for 
polynomial degree~$r = 8$.}
\label{fig:2d_dir_vs_gmres}
\end{figure}

Thus we solve the linear algebraic system $A x = b$ with GMRES using the 
mean value preconditioner $A_0 = I_{m+1} \otimes S_0$ from~\eqref{eqn:bar_A} 
as right preconditioner, that is, we solve
\begin{equation} \label{eqn:preconditioned_2d}
A A_0^{-1} y = b, \quad A_0 x = y.
\end{equation} 
Here $S_0$ denotes the FD discretization of the Helmholtz equation with 
absorbing boundary conditions and deterministic 
wavenumber~\eqref{eqn:wavenumber_k123}; see 
Theorem~\ref{thm:discretization_2d_abs}.
The solution of linear systems with the preconditioner $A_0$ is 
implemented as described in Section~\ref{sect:gmres_1d}.
We solve~\eqref{eqn:preconditioned_2d} with full GMRES (no restarts), 
\verb|tol=1e-8| and \verb|maxit=200| for polynomial degrees up to $r = 1, 
\ldots, 8$ in the stochastic Galerkin method.
In contrast to the experiments in 1D in Section~\ref{sect:experiments1d}, 
preconditioned GMRES is significantly faster than the 
direct solution with MATLAB's `backslash' command 
\verb|x=A\b|; see Figure~\ref{fig:2d_dir_vs_gmres} (left panel).
The computation times for preconditioned GMRES include the computation of the 
$LU$-decomposition of a diagonal block of $A_0$.
Furthermore, the relative residual norms in GMRES for polynomial degree 
$r = 8$ are shown in Figure~\ref{fig:2d_dir_vs_gmres} (right panel).
The mean value CSL preconditioner $M_0$ has a similar block-diagonal structure 
to $A_0$ and performs similarly well; see Figure~\ref{fig:2d_dir_vs_gmres}.

\begin{figure}
{\centering
\includegraphics[width=0.5\linewidth]{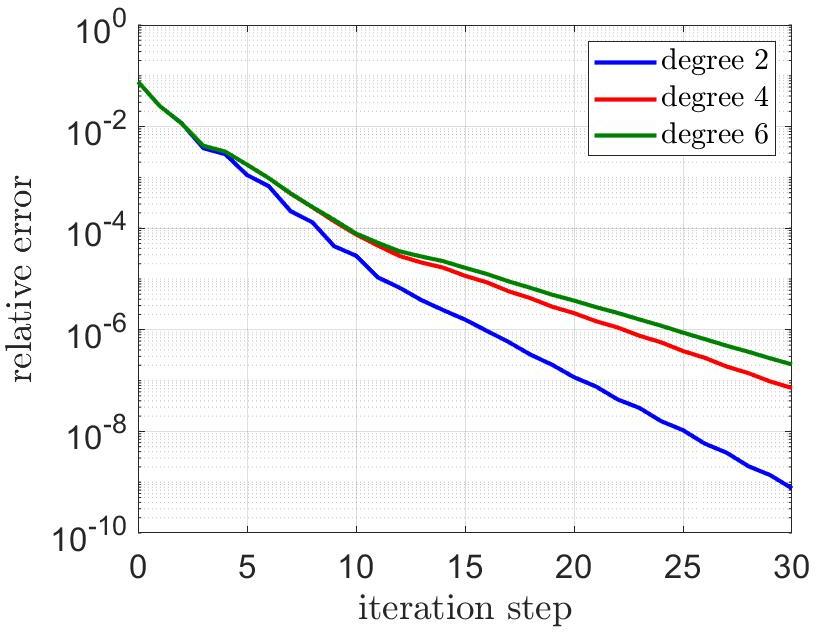}

}
\caption{Relative error norms in the stationary 
iteration~\eqref{eqn:stationary_iteration_A0} for different polynomial degrees 
$r$ in the stochastic Galerkin method.}
\label{fig:stationary_iteration}
\end{figure}

Alternatively to GMRES or a direct solution of the linear system, 
we also investigate the stationary iteration~\eqref{eqn:stationary_iteration} 
with $B = A_0$, i.e.,
\begin{equation} \label{eqn:stationary_iteration_A0}
A_0 x^{(i+1)} = b - (A-A_0) x^{(i)} \quad \text{for } i = 0,1,2, \ldots \, .
\end{equation}
We take the starting vector $x^{(0)} = A_0^{-1} b$. 
Linear systems with the matrix~$A_0$ are solved as described above.
For $\theta = 0.1$, this iteration converges.
Figure~\ref{fig:stationary_iteration} displays the relative error norms 
in the maximum-norm for polynomial degrees $r=2,4,6$, 
where we take the direct solution \verb|A\b| as the `exact' solution.
The slower convergence for larger degree~$r$ in the stochastic Galerkin method 
is expected, since the matrix size also grows causing higher condition numbers.
For $\theta = 0.2$, the stationary iteration diverges. 
This behavior is in agreement to Theorem~\ref{thm:matrix_asymptotic} 
and Corollary~\ref{cor:matrix_asymptotic}.

\begin{figure}[t!]
{\centering
\includegraphics[width=0.47\linewidth]
{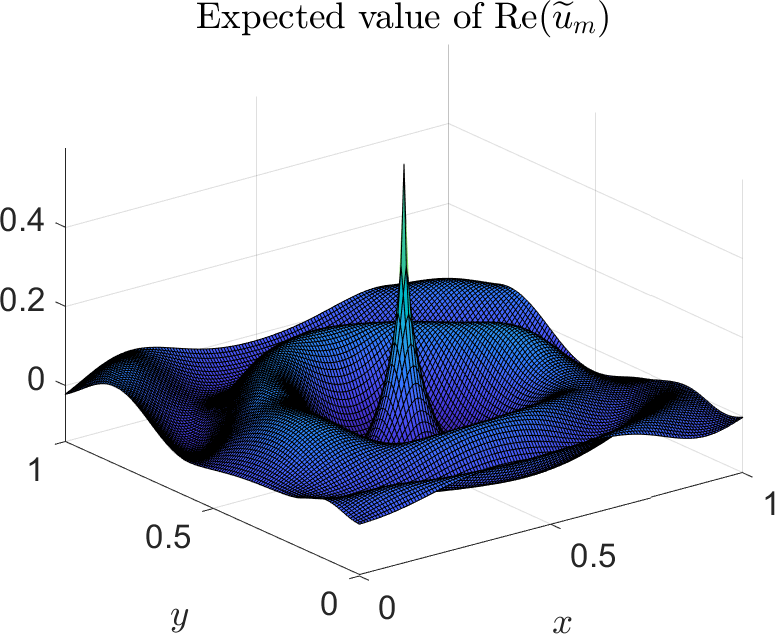}
\includegraphics[width=0.47\linewidth]
{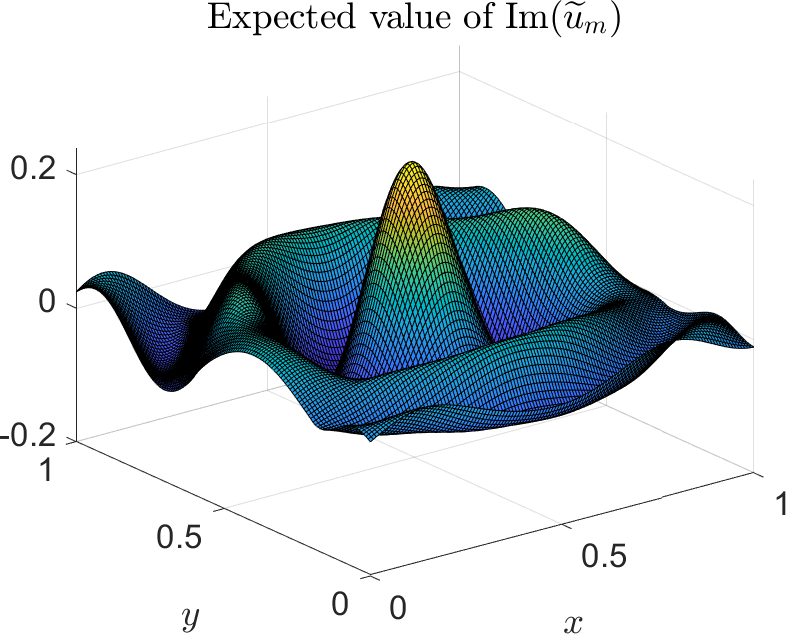}

\includegraphics[width=0.47\linewidth]
{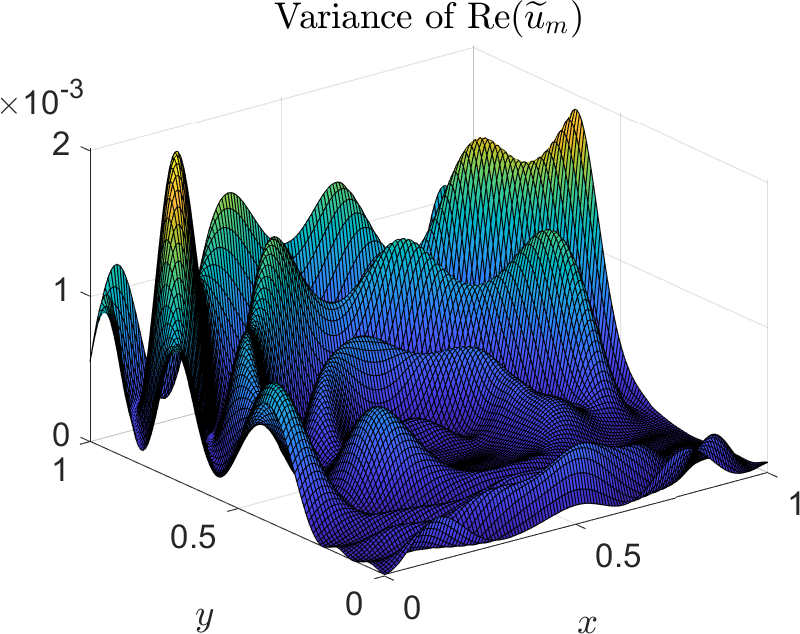}
\includegraphics[width=0.47\linewidth]
{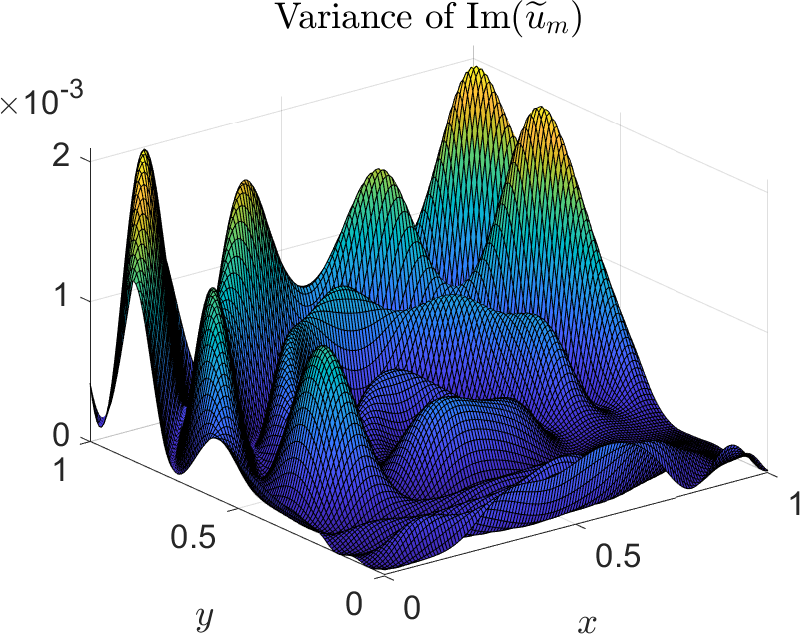}

}
\caption{Expected value (top) and variance (bottom) of $\re(\widetilde{u}_m)$ 
and $\im(\widetilde{u}_m)$.}
\label{fig:bE_Var}
\end{figure}

In Figure~\ref{fig:bE_Var}, the top row displays the expected value 
of the real and imaginary part of the computed stochastic Galerkin 
approximation $\widetilde{u}_m$ (with polynomial degree $r = 5$).
The variance is displayed in the bottom row of the figure.

\begin{figure}
{\centering
\includegraphics[width=0.47\linewidth]{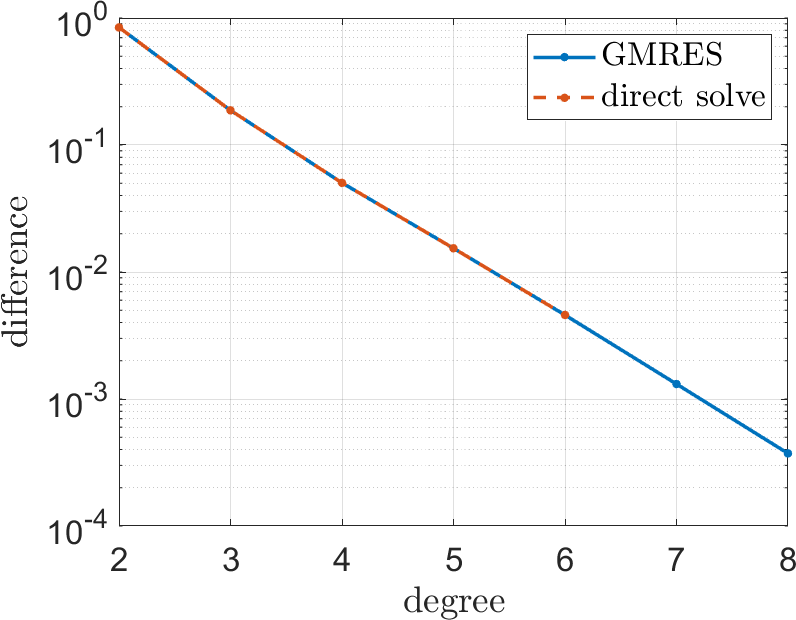}
\includegraphics[width=0.47\linewidth]{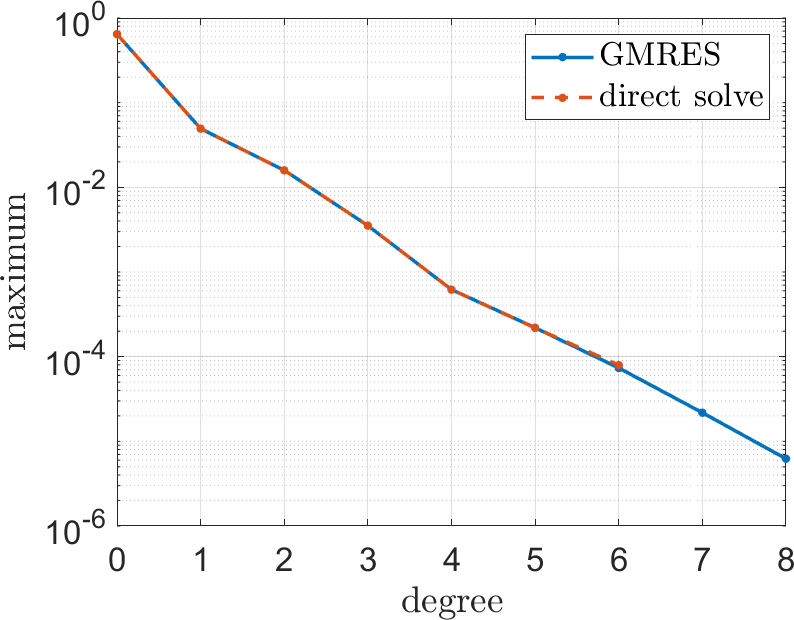}

}
\caption{Left: Euclidean norms $\norm{x_{r-1} - x_r}_2$ as a function of the 
polynomial degree $r = 2, \ldots, 8$, 
where $x_r$ is the solution of $A x = b$ using total degree~$r$ 
in the stochastic Galerkin method.
Right: Magnitudes~\eqref{eqn:magnitude_V} as a function of the degree~$j$.}
\label{fig:k123_convergence}
\end{figure}

Denote by $x_r$ the solution of $A x = b$ when using polynomials of degree up 
to $r$ in the stochastic Galerkin method, 
where the number of basis polynomials is given in~\eqref{eqn:number_p_ex}. 
The left panel of Figure~\ref{fig:k123_convergence} displays the differences 
$\norm{x_{r-1} - x_r}_2$ as a function of $r$ (the vector $x_{r-1}$ is padded 
with zeros at the end to match the size of $x_r$).
Their exponential decay suggests convergence of the stochastic Galerkin method.

Next, we fix the degree $r = 8$ in the stochastic Galerkin method.  Recall 
from~\eqref{eqn:Galerkin_approx} and~\eqref{eqn:LGS_from_FD_G} that the 
solution of $A x = b$ contains the coefficient vectors $V_0, \ldots, V_m$ of 
the polynomials $\phi_0, \ldots, \phi_m$ in the stochastic Galerkin method.
We also examine the largest maximum norm of the coefficients 
associated to polynomials of total degree (exactly)~$j$, 
i.e., the values
\begin{equation} \label{eqn:magnitude_V} 
\gamma_j = \max \left\{ \norm{V_i}_{\infty} : \, 
\deg(\phi_i) = j \right\} .
\end{equation}
The right panel of Figure~\ref{fig:k123_convergence} shows 
the magnitudes~\eqref{eqn:magnitude_V} for $j=0,1,\ldots,8$.
The observed exponential decay stems from the exponential convergence 
of~\eqref{eqn:pc}, since the wavenumber 
in~\eqref{eqn:stochastic_wavenumber_k123} is analytic in $\xi_1, \xi_2, \xi_3$.

We repeat this experiment with $\theta = 0.2$ instead of $\theta = 0.1$.
Overall, the behavior is similar as for $\theta = 0.1$, but convergence is 
slower:  the relative residual norms reach the prescribed tolerance in $60$ 
instead of $20$ iteration steps,
and also $\norm{x_{r-1} - x_r}_2$ as well as the 
magnitudes~\eqref{eqn:magnitude_V} converge more slowly.


\section{Conclusions}
\label{sect:conclusions}
We investigated the Helmholtz equation including a random wavenumber. 
The combination of a stochastic Galerkin method and a finite difference 
method yielded a high-dim\-en\-sional linear system of algebraic equations. 
We examined the iterative solution of these linear systems using 
three types of preconditioners: 
a complex shifted Laplace preconditioner, 
a mean value preconditioner, and 
a combined variant. 
Theoretical properties of the preconditioned linear systems were shown. 
Numerical computations demonstrate that the straightforward 
mean value preconditioner provides the most efficient iterative solution 
within these types.


\appendix
\section{Discretizations}
\label{sect:discretization}

\subsection{Finite differences and stochastic Galerkin method in 1D}

Consider the grid of equispaced points
\begin{equation} \label{eqn:grid_1d}
x_j = j h, \quad j = 0, 1, \ldots, q+1,
\end{equation}
in $\overline{Q} = \cc{0, 1}$ with mesh-size $h = 1/(q+1)$.
For brevity of notation, set
\begin{equation}
u_j(\xi) \coloneq u(x_j, \xi), \quad
k_j(\xi) \coloneq k(x_j, \xi), \quad
f_j \coloneq f(x_j), \quad j = 0, 1, \ldots, q+1.
\end{equation}
A finite difference discretization of the Helmholtz 
equation~\eqref{eqn:stochastic_helmholtz} using second order 
central differences yields
\begin{equation}
\frac{1}{h^2} \Bigl( -u_{j-1}(\xi) + 2 u_j(\xi) - u_{j+1}(\xi) \Bigr)
- k_j(\xi)^2 u_j(\xi) = f_j, \quad j = 1, \ldots, q.
\end{equation}
This discretization is consistent of order two.
Since $u_0(\xi) = u_{q+1}(\xi) = 0$ for all $\xi \in \Xi$ in the case of 
homogeneous Dirichlet boundary conditions, we obtain the following 
discretization.

\begin{theorem} \label{thm:discretization_1d_dir}
In the above notation, the Helmholtz equation~\eqref{eqn:stochastic_helmholtz} 
on $Q = \oo{0, 1}$ with homogeneous Dirichlet boundary conditions has the 
second order FD discretization
\begin{equation}
S(\xi) U(\xi) = F_0, \quad
U(\xi) = \begin{bmatrix} u_1(\xi) \\ \vdots \\ u_q(\xi) \end{bmatrix} \in \R^q, 
\quad
F_0 = \begin{bmatrix} f_1 \\ \vdots \\ f_q \end{bmatrix} \in \R^q,
\end{equation}
with the matrix
\begin{equation} \label{eqn:1d_dir_S}
S(\xi) = T - D(\xi) \in \R^{q,q},
\end{equation}
where $T$ is the discretization of the Dirichlet Laplacian,
\begin{align}
T &= \frac{1}{h^2}
\begin{bmatrix}
2 & -1 \\
-1 & 2 & -1 \\
& \ddots & \ddots & \ddots \\
& & -1 & 2 & -1 \\
& & & -1 & 2
\end{bmatrix}, \label{eqn:1d_dir_T} \\
D(\xi) &= \diag \big( k_1(\xi)^2, \ldots, k_q(\xi)^2 \big).
\end{align}
The matrices $T$ and $D(\xi)$ are symmetric positive definite.

Moreover, the coefficient vectors of the stochastic Galerkin 
approximation~\eqref{eqn:Galerkin_approx} are solutions of the linear algebraic 
system
\begin{equation}
A V = \begin{bmatrix} F_0 \\ \vdots \\ F_m \end{bmatrix}, \quad
V = \begin{bmatrix} V_0 \\ \vdots \\ V_m \end{bmatrix},
\end{equation}
where $F_i = 0 \in \R^q$ for $i = 1, \ldots, m$, and
\begin{equation} \label{eqn:1d_dir_A}
A = I_{m+1} \otimes T - \begin{bmatrix} C_{ij} \end{bmatrix} \in \R^{(m+1) q, 
(m+1) q},
\end{equation}
where
\begin{equation} \label{eqn:1d_dir_Cij}
C_{ij} \coloneq \sprod{D(\xi) \phi_j(\xi), \phi_i(\xi)}
= \diag \Big( \sprod{k_1(\xi)^2 \phi_j(\xi), \phi_i(\xi)}, 
\ldots, \sprod{k_q(\xi)^2 \phi_j(\xi), \phi_i(\xi)} \Big)
\end{equation}
for $i, j = 0, \ldots, m$.
The matrices $I_{m+1} \otimes T$ and $\begin{bmatrix} C_{ij} \end{bmatrix}$ are 
symmetric positive definite.
\end{theorem}

\begin{proof}
The matrix $T$ is symmetric positive definite 
by~\cite[Lem.~6.1]{GriffithsDoldSilvester2015} and $D(\xi)$ is symmetric 
positive definite since $k(x, \xi) > 0$ for all $x$ and $\xi$ by assumption.

The coefficient vectors $V_0, \ldots, V_m$ are determined from the 
orthogonality condition~\eqref{eqn:Galerkin_orthogonality}, i.e.,
$\sprod{S(\xi) \widetilde{U}_m(\xi), \phi_i(\xi)} = \delta_{i,0} F_0$, $i = 0, 
1, \ldots, m$.
Inserting $\widetilde{U}_m(\xi)$ from~\eqref{eqn:Galerkin_approx} and $S(\xi)$ 
from~\eqref{eqn:1d_dir_S} in the left hand side yields
\begin{align}
\sprod{S(\xi) \widetilde{U}_m(\xi), \phi_i(\xi)}
&= \sum_{j=0}^m \sprod{\phi_j(\xi) T V_j, \phi_i(\xi)} - 
\sum_{j=0}^m \sprod{\phi_j(\xi) D(\xi) V_j, \phi_i(\xi)} \\
&= T V_i - \sum_{j=0}^m C_{ij} V_j,
\end{align}
see Lemma~\ref{lem:galerkin_projection} and 
Corollary~\ref{cor:constant_galerkin_projection}, which shows that $A$ has the 
form~\eqref{eqn:1d_dir_A}.  Moreover, since $T$ and $D(\xi)$ are symmetric 
positive definite, also $I_{m+1} \otimes T$ and $\begin{bmatrix} C_{ij} 
\end{bmatrix}$ are symmetric positive definite by 
Lemma~\ref{lem:galerkin_projection}.
\end{proof}

The absorbing boundary conditions for $Q = \oo{0, 1}$ are
\begin{equation} \label{eqn:bc_absorbing}
-u'(0, \xi) - \ii k(0,\xi) u(0, \xi) = 0, \quad u'(1, \xi) - \ii k(1, \xi) u(1, 
\xi) = 0 \quad \text{for } \xi \in \Xi.
\end{equation}
We obtain a second order approximation of $u'$ in the boundary points as 
described in~\cite[Sect.~6.4.1]{GriffithsDoldSilvester2015}.
A Taylor expansion in $x_0 = 0$ yields
\begin{align}
u(x_1, \xi)
&= u(x_0, \xi) + h u'(x_0, \xi) + \frac{h^2}{2} u''(x_0, \xi) + O(h^3) \\
&= u(x_0, \xi) + h u'(x_0, \xi) - \frac{h^2}{2} \Big( k(x_0, \xi) u(x_0, \xi) + 
f(x_0) \Big) + O(h^3), \label{eqn:FD_first_derivative_bdry}
\end{align}
where we replaced $u''$ using the Helmholtz equation. This yields a second 
order approximation of $u'(x_0, \xi)$.  Inserting it 
in~\eqref{eqn:bc_absorbing} and dividing by $h$ yields the discretization
\begin{equation}
\frac{u_0(\xi) - u_1(\xi)}{h^2} - \frac{\ii}{h} k_0(\xi) u_0(\xi)
- \frac{k_0(\xi)^2}{2} u_0(\xi) = \frac{f_0}{2}.
\end{equation}
Similarly, the absorbing boundary condition in $x_{q+1} = 1$ is discretized by
\begin{equation}
\frac{-u_q(\xi) + u_{q+1}(\xi)}{h^2} - \frac{\ii}{h} k_{q+1}(\xi) 
u_{q+1}(\xi) - \frac{k_{q+1}(\xi)^2}{2} u_{q+1}(\xi) = \frac{f_{q+1}}{2}.
\end{equation}
The approximation is consistent of order two.
It leads to the following discretization.

\begin{theorem} \label{thm:discretization_1d_abs}
In the above notation, the Helmholtz equation~\eqref{eqn:stochastic_helmholtz} 
on $Q = \oo{0, 1}$ with absorbing boundary conditions has the 
second order FD discretization
\begin{equation}
S(\xi) U(\xi) = F_0, \quad
U(\xi) = \begin{bmatrix} u_0(\xi) \\ u_1(\xi) \\ \vdots \\ u_q(\xi) \\ 
u_{q+1}(\xi) \end{bmatrix} \in \R^{q+2}, \quad
F_0 = \begin{bmatrix} f_0/2 \\ f_1 \\ \vdots \\ f_q \\ 
f_{q+1}/2 \end{bmatrix} \in \R^{q+2},
\end{equation}
with the matrix
\begin{equation} \label{eqn:1d_abs_S}
S(\xi) = T - \ii D_1(\xi) - D_2(\xi) \in \C^{q+2, q+2}
\end{equation}
and real
\begin{align} 
T &= \frac{1}{h^2}
\begin{bmatrix}
1 & -1 \\
-1 & 2 & -1 \\
& \ddots & \ddots & \ddots \\
& & -1 & 2 & -1 \\
& & & -1 & 1
\end{bmatrix}, \label{eqn:1d_abs_T} \\
D_1(\xi) &= \frac{1}{h} \diag \Big( k_0(\xi), 0, \ldots, 0, k_{q+1}(\xi) \Big), 
\\
D_2(\xi) &= \diag \bigg( \frac{k_0(\xi)^2}{2}, k_1(\xi)^2, \ldots, k_q(\xi)^2, 
\frac{k_{q+1}(\xi)^2}{2} \bigg).
\end{align}
The matrices $T, D_1(\xi)$ are symmetric positive semidefinite (for all $\xi 
\in \Xi$), and $\ker(T) = \spann \{ [1, \ldots, 1]^\top \}$.
The matrix $D_2(\xi)$ is symmetric positive definite for all $\xi \in \Xi$.

Moreover, the coefficient vectors of the stochastic Galerkin 
approximation~\eqref{eqn:Galerkin_approx} are solutions of the linear algebraic 
system
\begin{equation} \label{eqn:1d_abs_A}
A V = \begin{bmatrix} F_0 \\ \vdots \\ F_m \end{bmatrix}, \quad V = 
\begin{bmatrix} V_0 \\ \vdots \\ V_m \end{bmatrix},
\end{equation}
where $F_i = 0 \in \R^{q+2}$ for $i = 1, \ldots, m$, and
\begin{equation}
A = I_{m+1} \otimes T - \ii [B_{ij}] - [C_{ij}] \in \C^{(m+1) (q+2), (m+1) 
(q+2)}
\end{equation}
with
\begin{align} \label{eqn:1d_abs_Bij}
B_{ij} &\coloneq
\frac{1}{h} \diag \Bigl( \sprod{k_0 \phi_j, \phi_i}, 0, \ldots, 0, 
\sprod{k_{q+1} \phi_j, \phi_i} \Bigr), \\
C_{ij} &\coloneq \diag \Bigl(
\frac{1}{2} \sprod{k_0^2 \phi_j, \phi_i}, \sprod{k_1^2 \phi_j, \phi_i},
\ldots, \sprod{k_q^2 \phi_j, \phi_i},
\frac{1}{2} \sprod{k_{q+1}^2 \phi_j, \phi_i} \Bigr)
\end{align}
for $i, j = 0, \ldots, m$.  Note that $B_{ij}, C_{ij} \in \R^{q+2, q+2}$.
The matrices $I_{m+1} \otimes T$ and $\begin{bmatrix} B_{ij} \end{bmatrix}$ are 
symmetric positive semidefinite, the matrix $\begin{bmatrix} C_{ij} 
\end{bmatrix}$ is symmetric positive definite.
\end{theorem}

\begin{proof}
The form of $S(\xi)$ in~\eqref{eqn:1d_abs_S} follows from the finite difference 
discretization described above.
To show that $T$ is symmetric positive definite, 
we make use of the Sturm sequence property of $-h^2 T$, which is a Jacobi 
matrix (real, symmetric, tridiagonal, with positive off-diagonal elements).
For $j = 1, \ldots, q+2$, denote by $T_j$ the upper left $j \times j$ block of 
$h^2 T$.  Then $\det(T_1) = 1$, $\det(T_2) = 1$ and, by induction, $\det(T_j) = 
2 \det(T_{j-1}) - \det(T_{j-2})$ for $j = 3, \ldots, q+1$, which shows 
$\det(T_j) = 1$ for $j = 1, \ldots, q+1$.  Finally, $\det(T_{q+2}) = 
\det(T_{q+1}) - \det(T_q) = 0$.
Thus, the determinants $\det(-T_j) = (-1)^j$ alternate for $j = 1, \ldots, 
q+1$ while $\det(-T_{q+2}) = 0$, so that there is only one ``agreement in 
sign'' from the final $0$ in the Sturm sequence for $-h^2 T$, showing that $- 
h^2 T$ has exactly one nonnegative eigenvalue, namely $0$.  Therefore, $T$ has 
one eigenvalue $0$ and all other eigenvalues of $T$ are positive.
The rest is very similar to the proof of 
Theorem~\ref{thm:discretization_1d_dir}.
\end{proof}

\begin{lemma} \label{lem:discretization_1d_abs_bandwith}
In the notation of Theorem~\ref{thm:discretization_1d_abs}, if $\xi$ is a 
random variable that is uniformly distributed in $\cc{-1, 1}$
and if $k(x, \xi)$ is a polynomial in $\xi$ of degree at most $n$ for all $x 
\in Q$, then $B_{ij} = 0$ for $\abs{i - j} > n$ and $C_{ij} = 0$ for $\abs{i - 
j} > 2n$.
\end{lemma}

\begin{proof}
The proof relies on the fact of orthogonal polynomials, that $\sprod{p, \phi_j} 
= 0$ for all polynomials with $\deg(p) < j$.
If $\deg_\xi(k(x,\xi)) \leq n$, then $\sprod{k(x, \xi) \phi_i(\xi), \phi_j} = 
0$ for $i + n < j$, i.e., $i - j < -n$.  Since $k$ is real, we also have
$\sprod{k(x, \xi) \phi_i(\xi), \phi_j} = \sprod{\phi_i(\xi), k(x, \xi) \phi_j} 
=  0$ for $n + j < i$, i.e., $i - j > n$.  This shows that $B_{ij} = 0$ for 
$\abs{i - j} > n$.
Similarly $C_{ij} = 0$ for $\abs{i - j} > 2n$ since $\deg_\xi(k(x, \xi)^2) = 
2n$.
\end{proof}

If the wavenumber is constant in space, the matrices $[B_{ij}]$ and $[C_{ij}]$ 
simplify, as indicated in the next lemma.

\begin{lemma} \label{lem:discretization_1d_abs_k_xi}
In the notation of Theorem~\ref{thm:discretization_1d_abs}, if the wavenumber 
is constant in space, i.e., $k(x, \xi) = k(\xi)$, then
\begin{align}
B_{ij} &= \sprod{k(\xi) \phi_j(\xi), \phi_i(\xi)} D_1, \quad 
D_1 = \frac{1}{h} \diag \big( 1, 0, \ldots, 0, 1 \big), \\
C_{ij} &= \sprod{k(\xi)^2 \phi_j(\xi), \phi_i(\xi)} D_2, \quad D_2 = \diag 
\Bigr( \frac{1}{2}, 1, \ldots, 1, \frac{1}{2} \Bigr),
\end{align}
so that
\begin{equation}
[B_{ij}] = [\sprod{k(\xi) \phi_j(\xi), \phi_i(\xi)}]_{ij} \otimes D_1, \quad
[C_{ij}] = [\sprod{k(\xi)^2 \phi_j(\xi), \phi_i(\xi)}]_{ij} \otimes D_2.
\end{equation}
\end{lemma}

Lemma~\ref{lem:discretization_1d_abs_bandwith} and
Lemma~\ref{lem:discretization_1d_abs_k_xi} also hold in the setting of 
Theorem~\ref{thm:discretization_1d_dir} with the obvious modifications.

\subsection{Finite differences and stochastic Galerkin method in 2D}

We discretize the stochastic Helmholtz 
equation~\eqref{eqn:stochastic_helmholtz} 
on $Q = \oo{0, 1}^2$.
Let $q \in \N$.  We discretize $\cc{0, 1}^2$ by the grid
\begin{equation} \label{eqn:grid_2d}
(x_i, y_j) = (i h, j h), \quad i, j = 0, 1, \ldots, q+1,
\end{equation}
with mesh-size $h = 1/(q+1)$.
For brevity of notation, set
\begin{equation}
u_{i,j} = u(x_i, y_j, \xi), \quad
k_{i,j} = k(x_i, y_j, \xi), \quad
f_{i,j} = f(x_i, y_j), \quad i, j = 0, 1, \ldots, q+1.
\end{equation}
Discretizing the Laplacian with the $5$-point stencil leads to
\begin{equation} \label{eqn:de_discretized}
\frac{1}{h^2} (- u_{i-1, j} + 2 u_{i,j} - u_{i+1, j} - u_{i, j-1} + 
2 u_{i,j} - u_{i, j+1}) - k_{i,j}^2 u_{i,j} = f_{i,j}
\end{equation}
for $i, j = 1, \ldots, N$.
This discretization is consistent of order two.

\begin{theorem} \label{thm:discretization_2d_dir}
In the above notation, the stochastic Helmholtz 
equation~\eqref{eqn:stochastic_helmholtz} on $Q = \oo{0, 1}^2$ with homogeneous 
Dirichlet boundary conditions has the second order FD discretization
\begin{equation}
S(\xi) U(\xi) = b_0,
\end{equation}
where the function values are ordered as
\begin{align}
U(\xi) &= [u_{1,1}, u_{2,1}, \ldots, u_{q,1}, u_{1,2}, \ldots, u_{q,2}, \ldots, 
u_{1,q}, \ldots, u_{q,q}]^\top \in \R^{q^2}, \quad \xi \in \Xi, \\
b_0 &= [f_{1,1}, f_{2,1}, \ldots, f_{q,1}, f_{1,2}, \ldots, f_{q,2}, \ldots, 
f_{1,q}, \ldots, f_{q,q}]^\top \in \R^{q^2},
\end{align}
the matrix is given by
\begin{equation}
S(\xi) = L - D_2(\xi) \in \R^{q^2, q^2},
\end{equation}
and
\begin{align}
L &= I_q \otimes T + T \otimes I_q, \\
D_2(\xi) &= \diag(k_{1,1}^2, \ldots, k_{q,1}^2, k_{1,2}^2, \ldots, k_{q,2}^2, 
\ldots, k_{1,q}^2, \ldots, k_{q,q}^2)
\end{align}
with $T$ from~\eqref{eqn:1d_dir_T}.
The matrices $L$ and $D_2(\xi)$ are symmetric positive definite.

Moreover, the coefficient vectors of the stochastic Galerkin 
approximation~\eqref{eqn:Galerkin_approx} are solutions of the linear algebraic 
system
\begin{equation}
A V = b, \quad
V = \begin{bmatrix} V_0 \\ \vdots \\ V_m \end{bmatrix}, \quad
b = \begin{bmatrix} b_0 \\ \vdots \\ b_m \end{bmatrix},
\end{equation}
where $b_i = 0 \in \R^{q^2}$ for $i = 1, \ldots, m$, and
\begin{equation} \label{eqn:2d_dir_A}
A = I_{m+1} \otimes L - \begin{bmatrix} C_{ij} \end{bmatrix} \in \R^{(m+1)q^2, 
(m+1)q^2},
\end{equation}
where
\begin{equation} \label{eqn:2d_dir_Cij}
C_{ij} \coloneq \diag \Bigl( 
\sprod{k_{1,1}^2 \phi_i, \phi_j}, \ldots, \sprod{k_{q,1}^2 \phi_i, \phi_j}, 
\ldots, \sprod{k_{1,q}^2 \phi_i, \phi_j}, \ldots, \sprod{k_{q,q}^2 \phi_i, 
\phi_j} \Bigr)
\end{equation}
for $i, j = 0, \ldots, m$.
The matrices $I_{m+1} \otimes L$ and $\begin{bmatrix} C_{ij} \end{bmatrix}$ are 
symmetric positive definite.
\end{theorem}

To obtain a second order discretization of absorbing boundary conditions, we 
proceed as described in~\cite[Sect.~10.2.1]{GriffithsDoldSilvester2015}.
This leads to the following result.

\begin{theorem} \label{thm:discretization_2d_abs}
In the above notation, the Helmholtz equation~\eqref{eqn:stochastic_helmholtz} 
on $Q = \oo{0, 1}^2$ with absorbing boundary conditions has the second order FD 
discretization
\begin{equation}
S(\xi) U(\xi) = b_0,
\end{equation}
where the right hand side is
\begin{equation}
b_0 = \begin{bmatrix} \frac{1}{2} F_0 \\ F_1 \\ \vdots \\ F_N \\ \frac{1}{2} 
F_{q+1} \end{bmatrix} \in \R^{(q+2)^2}, \quad
F_j = \begin{bmatrix} \frac{1}{2} f_{0,j} \\ f_{1,j} \\ \vdots \\ f_{q,j} \\ 
\frac{1}{2} f_{q+1,j} \end{bmatrix} \in \R^{q+2}, \quad 0 \leq j \leq q+1,
\end{equation}
and where the matrix is given by
\begin{equation} \label{eqn:2d_abs_S}
S(\xi) = L - \ii D_1(\xi) - D_2(\xi) \in \C^{(q+2)^2, (q+2)^2}
\end{equation}
with block matrices
\begin{align}
L &= D \otimes T + T \otimes D, \\
D_1(\xi) &= \frac{1}{h} \diag \Bigl( D_{1,0}(\xi), D_{1,1}(\xi), \ldots, 
D_{1,q}(\xi), D_{1,q+1}(\xi) \Bigr), \\
D_2(\xi) &= \diag \Bigl(\frac{1}{2} D_{2,0}(\xi), D_{2,1}(\xi), \ldots, 
D_{2,q}(\xi), \frac{1}{2} D_{2,q+1}(\xi) \Bigr),
\end{align}
and $(q+2) \times (q+2)$-blocks $T$ from~\eqref{eqn:1d_abs_T},
\begin{align}
D &= \diag \Bigl( \frac{1}{2}, 1, \ldots, 1, \frac{1}{2} \Bigr), \\
D_{1,j}(\xi) &=
\begin{cases}
\diag(k_{0,j}, k_{1,j}, \ldots, k_{q,j}, k_{q+1,j}), & j = 0, q+1, \\
\diag(k_{0,j}, 0, \ldots, 0, k_{q+1,j}), & j = 1, \ldots, q,
\end{cases} \\
D_{2,j}(\xi) &= \diag \Big( \frac{1}{2} k_{0,j}^2, k_{1,j}^2, \ldots, 
k_{q,j}^2, \frac{1}{2} k_{q+1,j}^2 \Big), \quad j = 0, 1, \ldots, q+1.
\end{align}
The matrices $L$, $D_1(\xi)$ are symmetric positive semidefinite (for all $\xi 
\in \Xi$).
The matrix $D_2(\xi)$ is symmetric positive definite for all $\xi \in \Xi$.

Moreover, the coefficient vectors of the stochastic Galerkin 
approximation~\eqref{eqn:Galerkin_approx} are solutions of the linear algebraic 
system
\begin{equation} \label{eqn:2d_abs_A}
A V = b, \quad
b = \begin{bmatrix} b_0 \\ \vdots \\ b_m \end{bmatrix}, \quad
V = \begin{bmatrix} V_0 \\ \vdots \\ V_m \end{bmatrix},
\end{equation}
where $b_i = 0 \in \R^{(q+2)^2}$ for $i = 1, \ldots, m$, and
\begin{equation}
A = I_{m+1} \otimes L - \ii [B_{ij}] - [C_{ij}] \in \C^{(m+1) (q+2)^2, (m+1) 
(q+2)^2}
\end{equation}
with
\begin{equation} \label{eqn:2d_abs_Bij}
B_{ij} \coloneq \sprod{D_1(\xi) \phi_i(\xi), \phi_j(\xi)}, \quad
C_{ij} \coloneq \sprod{D_2(\xi) \phi_i(\xi), \phi_j(\xi)} \in \R^{(q+2)^2, 
(q+2)^2}
\end{equation}
for $i, j = 0, \ldots, m$.
The matrices $I_{m+1} \otimes L$ and $\begin{bmatrix} B_{ij} \end{bmatrix}$ are 
symmetric positive semidefinite, the matrix $\begin{bmatrix} C_{ij} 
\end{bmatrix}$ is symmetric positive definite.
\end{theorem}

\subsection{Point sources}

The source term $f$ in the Helmholtz equation is often a point source, 
typically represented by a Dirac delta distribution, say $f(x) = \delta(x - 
a)$.  In our finite difference approximation in one dimension, 
we discretize $f$ by $1/h$ at $a$ (or at a grid point with smallest distance to 
$a$) and $0$ at the other grid points; see also~\cite{WangJungBiondini2014} for 
a discussion of the discretization of the Dirac distribution.
In two space dimension, we discretize $f(x) = \delta(x-a)$ by $1/h^2$ at~$a$ 
(or at a closest grid point).

Specifically for $f(x) = \delta(x - 1/2)$ and the grid~\eqref{eqn:grid_1d}
or $\delta( (x,y) - (1/2, 1/2))$ and the grid~\eqref{eqn:grid_2d}, let 
$t = \lceil q/2 \rceil = \lfloor (q+1)/2 \rfloor$.
If $q$ is odd, $x_t = 1/2$ is the exact midpoint.  If $q$ is even, $t = q/2$ 
and $x_t = 1/2 - h/2$ is the smaller of the two grid points closest to $1/2$.
We then discretize $f$ by
\begin{equation}
f(x_j) = \begin{cases} 1/h, & j = t, \\ 0, & j \neq t, \end{cases}
\quad \text{or} \quad
f(x_j, y_\ell) = \begin{cases} 1/h^2, & j = \ell = t, \\ 0, & 
\text{else.} 
\end{cases}
\end{equation}

\subsection{Stochastic Galerkin and finite differences in 1D}
\label{sect:Galerkin_FDM_1d}

We discretize the deterministic system of PDEs obtained from the stochastic 
Helmholtz equation with the the stochastic Galerkin method as described in 
Section~\ref{sect:Galerkin_FDM}.

We first discretize the PDE~\eqref{eqn:de_v} with the grid~\eqref{eqn:grid_1d}.
We have
\begin{equation} \label{eqn:de_vj}
- v_j''(x) - \sum_{i = 0}^m c_{ij}(x) v_i(x) = F_j(x), \quad x \in Q = 
\oo{0, 1}, \quad j = 0, 1, \ldots, m.
\end{equation}
Second order central differences yield the approximation
\begin{equation} \label{eqn:de_v_discretized}
\frac{1}{h^2} (-v_j(x_{\ell-1}) + 2 v_j(x_\ell) - v_j(x_{\ell+1})) - \sum_{i = 
0}^m c_{ij}(x_\ell) v_i(x_\ell) = F_j(x_\ell), \quad \ell = 1, \ldots, q.
\end{equation}
Given homogeneous Dirichlet boundary conditions~\eqref{eqn:bc_v_dirichlet}, we 
order the unknowns as
\begin{equation}
V = [v_0(x_1), \ldots, v_0(x_q), v_1(x_1), \ldots, v_1(x_q), \ldots, v_m(x_1), 
\ldots, v_m(x_q)]^\top.
\end{equation}
Then~\eqref{eqn:de_v_discretized} yields the block system~\eqref{eqn:1d_dir_A}.
The boundary conditions~\eqref{eqn:bc_v} are
\begin{equation} \label{eqn:bc_vj_1d}
-v_j'(0) - \ii \sum_{i = 0}^m b_{ij}(0) v_i(0) = 0, \quad
v_j'(1) - \ii \sum_{i = 0}^m b_{ij}(1) v_i(1) = 0, \quad 
j = 0, 1, \ldots, m.
\end{equation}
A second order discretization of $v_j'$ is derived as 
in~\eqref{eqn:FD_first_derivative_bdry}.
Ordering the unknowns as
\begin{equation}
V = [v_0(x_0), \ldots, v_0(x_{q+1}), 
v_1(x_0), \ldots, v_1(x_{q+1}), 
\ldots,
v_m(x_0), \ldots, v_m(x_{q+1})]^\top
\end{equation}
yields the block system~\eqref{eqn:1d_abs_A}.


\small
\bibliographystyle{siam}
\bibliography{helmholtz.bib}

\end{document}